\documentclass[12pt]{amsart}
\usepackage[]{hyperref}
\usepackage{txfonts}
\usepackage{cite}
\usepackage{mathdots}
\usepackage{amssymb}
\usepackage{amsfonts}
\usepackage{amsbsy}
 \usepackage{graphicx}
 \usepackage{ulem}
\usepackage{amscd}
\usepackage{graphicx}
\usepackage{mathrsfs}
\usepackage{amssymb}
\usepackage{latexsym}
\usepackage{amsmath}
\usepackage{amsfonts}
\usepackage{amsbsy}
\usepackage{amscd}
\usepackage{subfigure}
\usepackage{verbatim}
\usepackage[dvipsnames,usenames]{color} 
\usepackage{color,tikz}
\usetikzlibrary{calc}
\usepackage{amsfonts}
\usepackage{amsbsy,calc}
 \usepackage{graphicx,ifthen,cite}

 \def\cV{\mathcal{V}}\def\cE{\mathcal{E}}

\newcommand{\cL}{{\mathcal L}}

\newcommand{\cT}{{\mathcal T}}

\renewcommand{\P}{{\mathcal P}}

\newcommand{\R}{{\mathbb R}}
\newcommand{\Z}{{\mathbb Z}}

\newcommand{\Q}{{\mathbb Q}}
\newcommand{\N}{{\mathbb N}}

\renewcommand{\v}{{\bf v}}

\newcommand{\ep}{\varepsilon}

\newtheorem{prop}{Proposition}[section]
\newtheorem{lem}[prop]{Lemma}

\newtheorem{coro}[prop]{Corollary}
\newtheorem{theo}[prop]{Theorem}

\newtheorem{exam}[prop]{Example}
\newtheorem{rema}[prop]{Remark}
\theoremstyle{definition}
\newtheorem{defi}[prop]{Definition}

 \def\cN{\mathcal{N}}

  \def\supp{{\rm supp}}

\def\bi{\mathbf{i}}
\def\bj{\mathbf{j}}

\newif\ifdraft
\drafttrue
\ifdraft
\else
\fi
\def\ppmod{\!\!\pmod}

\numberwithin{equation}{section} 

\begin{document}

\baselineskip=16pt


\title{Spectral Eigen-subspace and Tree Structure for a Cantor Measure}
\author[G. T. Deng]{Guotai Deng}
\address{School of Mathematics and Statistics \& Hubei Key Laboratory of Mathematical Sciences, Central China Normal University, Wuhan 430079, P. R. China.}
\email{hilltower@163.com}

\author[Y.-S. Fu]{Yan-Song Fu${}^*$}
\address{Department of  Mathematics   \\ China University of Mining and Technology (Beijing) \\ Beijing, 100083
\\P. R. China.}
\email{yansong$\_$fu@126.com;}

\author[Q. C. Kang]{Qingcan Kang}
 \address{Department of Industrial Engineering and Decision Analytics \\ The Hong Kong University of Science and Technology, Kowloon, Hong Kong SAR, P. R. China}
\email{qkangaa@connect.ust.hk}

\def\subjclassname{2020 Mathematics Subject Classification}
\subjclass[2020]{Primary 42A10,28A80, 42A65, }

    \date{\today}

    \begin{abstract}
    In this work we investigate the question of constructions of the possible Fourier bases $E(\Lambda)=\{e^{2\pi i \lambda x}:\lambda\in\Lambda\}$ for the Hilbert space $L^2(\mu_4)$,  where $\mu_4$ is the standard middle-fourth Cantor measure and $\Lambda$ is a countable discrete set. We show that  the set
    $$\mathop \bigcap_{p\in 2\Z+1}\left\{\Lambda\subset \R: \text{$E(\Lambda)$ and $E(p\Lambda)$ are Fourier bases for   $L^2(\mu_4)$}\right\}$$ has the cardinality of the continuum. We also give other characterizations on the orthonormal set of exponential functions   being  a basis for the space $L^2(\mu_4)$ from the viewpoint of measure and dimension. Moreover, we provide a method of  constructing explicit discrete set $\Lambda$ such that $E(\Lambda)$ and its all odd scaling sets  $E(\Lambda),p\in2\Z+1,$ are still Fourier bases for $L^2(\mu_4)$.
    \end{abstract}

    \maketitle
    \def\supp{\textrm{supp}}

     \section{Introduction}
     A discrete set $\Lambda \subseteq \R^d$   is called a {\it spectrum} for  a Borel probability measure $\mu$  on $\R^d$ if the system of exponential functions $$E(\Lambda)=\{e_\lambda(x):=e^{2 \pi i \langle \lambda,  x\rangle}: \lambda\in \Lambda\}$$ forms an orthonormal basis   or {\it Fourier basis} for the Hilbert space $L^2(\mu)$, where  $\langle \cdot,  \cdot\rangle$ denotes the standard inner product of $\R^d$. Also, $\mu$ is called a {\it spectral measure} and $(\mu, \Lambda)$ forms a {\it spectral pair}. It is well-known that the Fourier-Stieltjes transformation of the measure $\mu$ plays an important role in determining the spectral property of the measure $\mu$ and it is defined by
      $$\widehat{\mu}(\xi)=\int e^{-2\pi i \langle\xi, x\rangle}d\mu(x) \qquad (\xi\in\R^d).$$
      Obviously, there are two crucial ingredients for a discrete set $\Lambda$ to be a spectrum $\mu$:

     \begin{enumerate}
       \item (Orthogonality) for any two distinct elements $\lambda,\lambda'$ in $\Lambda$, $$\langle e_\lambda, e_{\lambda'} \rangle_{L^2(\mu)}=\int e^{2\pi i (\lambda-\lambda')x}d\mu(x)=\widehat{\mu}(\lambda'-\lambda),$$
       \item  (Completeness)  if $\langle f, e_{\lambda} \rangle_{L^2(\mu)}=0$ for all $\lambda\in \Lambda$, then $f=0$ holds $\mu$-almost everywhere.
     \end{enumerate}

     The existence of a spectrum for the measure is the fundamental problem  for the harmonic analysis of $\mu$. If $\mu$ is a Lebesgue measure, it is  closely  related to the famous {\it Fuglede conjecture}\cite{Fug1974} proposed in 1974, which says that $\chi_\Omega dx$ is a spectral measure if and only if $\Omega$ is a translational tile, where $\Omega$ is a positive Lebesgue measurable  subset of $\R^d$. One can refer to \cite{Fug1974,T2004,LM2021,KM2006}  and the references therein for more advances on this famous conjecture. Moreover, the spectral measures must be of {\it pure type}\cite{LW2006,HLL2013}: either discrete, or absolutely continuous, or singular continuous. A discrete spectral measure has only finitely  many atoms, the density function of absolutely continuous  spectral measure is a constant almost everywhere on its support \cite{L2011,DL2014}. Nevertheless, the general spectral properties of singular continuous measures are much less understood, see \cite{LW2002,LW2006,DL2014,L2018,DC2021} and references therein for details.

    The first non-atomic singular continuous spectral measure  was discovered by Jorgensen and Pedersen \cite{JP} in 1998, who proved that  the middle-fourth Cantor measure $\mu_4$, satisfying the invariance equation
    \begin{equation*}
    \mu_4(\cdot)=\dfrac12\mu_4\circ \tau_1^{-1}(\cdot)+\dfrac12\mu_4\circ \tau_2^{-1}(\cdot),
    \end{equation*}
    where $\{\tau_1(x)=4^{-1}x,  \tau_2(x)=4^{-1}(x+2)\}$, is a spectral measure with a canonical spectrum
      \begin{equation}\label{eq.intro.1}
 \Lambda_1=\left\{\sum_{k=1}^n 4^{k-1}a_k:\,a_k\in\{0,1\},n\in\N\right\},
 \end{equation}
    while the standard middle-third Cantor measure is not a spectral measure. After that, Jorgensen and Pedersen's discovery opened up a new field in developing  orthogonal harmonic analysis for fractal measures, and various new surprising singular phenomena different from that of Lebesgue measures were discovered by many researchers, one can refer to   \cite{Str2000,LW2002,Str2006,HL2008,D2012,DFY2023,DHL2019,FHL2015,AFL,AH2014,LW2002,LW2006,LLP2021,DC2021,LZWC2021,LMW2022} and the references therein. Among them, the most surprising thing is that the singular spectral measure has more than one spectrum including 0  \cite{Str2000}, and the Beurling dimension of spectra for some singular spectral measures  have the intermediate value property \cite{LW2023,LW2024}. Accordingly, if a spectral pair $(\mu, \Lambda)$ exists, as in classical Fourier analysis, Strichartz \cite{Str2000} firstly studied the convergence of Fourier series 
    $$\sum_{\lambda\in\Lambda} \langle  f, e_\lambda\rangle_{L^2(\mu)}e_\lambda,$$ and the recent research proved that, for certain Cantor spectral measure $\mu$, there is uncountably many spectra such that the associated mock Fourier series of any continuous function converges uniformly to itself, and the associated mock Fourier series of any $L^1(\mu)$ function converges to itself $\mu$-almost everywhere \cite{Str2006,FTW2022}. Surprisingly,  Dutkay, Han and Sun \cite{DHS2014} (also see \cite{PA2023}) found that discrete sets $p\Lambda_1, p=17,23,29,$ written as in \eqref{eq.intro.2}, are still spectra for $\mu_4$, but there is some continuous function such that the  associated Fourier series diverges at $0$. These findings imply that the spectra for singular continuous spectral measures are more intricate than the case of Lebesgue measures, and their convergence and divergence of Fourier series also depend  on the selection of different spectrum.  So we are forced to investigate the possible spectra for a given singular spectral measure.

     However, to the best of our knowledge,  there is no singular continuous spectral measure whose spectra have been completely characterized, even for simple spectral measure $\mu_4$, see \cite{DHS2009,DHL2013,Dai2016,LW2023,LW2024} and references therein.   The main purpose of this paper is to  find  new spectra for $\mu_4$  by investigating its spectral eigen-subspace  and the tree structure.

    \subsection{Spectral eigenvalue problem of $\mu_4$}
      The  typical phenomena for the spectrum of singular spectral measures  is that it may have certain {\it scaling} property. 
      For  $\mu_4$ and  Jorgensen-Pedersen's canonical spectrum $\Lambda_1$, one can easily check that $E(p\Lambda_1)$ forms an orthogonal set for $\mu_4$ if and only if $p\in2\Z+1$, where
    \begin{equation}\label{eq.intro.2}
    p\Lambda_1=p\left\{\sum_{k=1}^n 4^{k-1}a_k:\,a_k\in\{0,1\},n\ge1\right\}.
    \end{equation}  Thus, the basic but most important thing is to check for what odd integers $p$, the scaling set $p\Lambda_1$ is still a spectrum for $\mu_4$.  Strichartz\cite[p. 218]{Str2000} showed that $5\Lambda_1$ is  a spectrum for $\mu_4$, but Dutkay \textit{et al.} \cite{DHS2009} showed that $3\Lambda_1$ is not  a spectrum for $\mu_4$. Other characterization on odd integers were given in \cite{JKS2011,Li2011,LX2016,DJ2012,DH2016,Dai2016}, etc. This leads to  the {\it spectral eigenvalue problem} of a spectral measure on $\R$ (resp. {\it spectral eigenmatrix problem} on $\R^d$), see \cite{FHW2018,FH2017,F2019,ADH2022,LW2024}, etc.
    \begin{defi}
    A real number $p$ is called a \textit{spectral eigenvalue} of spectral measure $\mu$ on $\R$ if there is a spectrum $\Lambda$ such that its scaling set $p\Lambda$ is also a spectrum for $\mu$. In this case, the spectrum $\Lambda$ is called an {\it eigen-spectrum} corresponding to the spectral eigenvalue $p$, and the set $$\{\Lambda\subseteq \R: \text{$\Lambda$ and $p\Lambda$ are both spectra for $\mu$}\}$$ is called the \textit{spectral eigen-subspace} of $\mu$ corresponds to the  spectral eigenvalue $p$.
    \end{defi}
   The study on the spectral eigenvalues problem is of interest and helps us find more surprising facts. For example, the mock Fourier series associated to the eigen-spectra corresponding to the same eigenvalue may have the same convergence property, see \cite[Section 5]{FHW2018} and \cite{FTW2022}. Other connections  between the spectral eigenvalues problem and Fourier analysis, Ergodic theory, Operator theory, Number theory and Fractal dimension theory were found by Strichartz \cite{Str2000,Str2006}, Dutkay, Han and Sun \cite{DHS2014,DHSW2011}, Dutkay and Hausserman \cite{DH2016}, Jorgensen and his coauthors \cite{Jor2012,JKS2011,JKS2012}, Li and Wu \cite{LW2023,LW2024}, etc.

   It was proved \cite[Theorem 4.3]{FHW2018} that an integer $p \in \Z$ is a spectral eigenvalue of $\mu_4$ if and only if $p\in 2\Z+1$.  In this paper we  characterize the  \textit{common spectral eigen-subspace} of $\mu_4$ corresponding to spectral eigenvalue $p\in2\Z+1$.

     \begin{theo}\label{coro1}
   The set
    $$    \mathop \cap_{p\in 2\Z+1}\{\Lambda\subset \R: \text{$\Lambda$ and $p\Lambda$ are all spectra for the measure $\mu_4$}\}$$ has the cardinality of the continuum; i.e., there exist uncountably many discrete sets $\Lambda$  such that the sets $\Lambda$ and $p\Lambda$ are all spectra for the measure $\mu_4$ for all  $p\in 2\Z+1$.
        \end{theo}
    Theorem \ref{coro1} gives out the proof of the existence of {\it common eigen-spectra} corresponding to all spectral eigenvalues $p\in2\Z+1$. Such an explicit common eigen-spectrum is  constructed in Theorem \ref{thm.2}, which relies on an algebraic constructive method in Theorem \ref{lemkey}.

    \subsection{Tree structure of $\mu_4$}
    In  2009, Dutkay, Han and Sun \cite{DHS2009} was the first to    investigate  the construction of spectra for $\mu_4$ systematically. They found that the  exponential maximal orthogonal sets (EMOS) of $\mu_4$ has certain {\it tree} structure (e.g., see \cite[Figure 1]{DHS2009}), and established a  one-to-one correspondence between EMOS with the spectral labels of  infinite binary tree via the digit set $\{0,1,2,3\}$ (see \cite[Theorem 3.3]{DHS2009}). This idea was generalized by Dai \textit{et al} \cite{DHL2013,Dai2016} by using digit set $\{-1,0,1,2\}$ instead of digit set $\{0,1,2,3\}$. Also, \cite{DHS2009} provided us a new exotic phenomena different from Lebesgue measure: a maximal set of orthogonal exponentials may not necessarily be an orthonormal basis for $\mu_4$, so a remaining challenging problem  for $\mu_4$ is  to investigate \textit{when a (maximal) orthogonal set of exponentials will be a Fourier basis for $\mu_4$?}

    Although some sufficient or necessary conditions for maximal orthogonal set of exponentials being a spectrum were given in \cite{DHL2013,Dai2016,DHS2009,HTW}, etc, there are  other orthogonal sets,  generated by digits not necessarily $\{0,1,2,3\}$ or $\{-1,0,1,2\}$, are deserved to be studied. In this paper, we will study the completeness of the following discrete set  $E(\Lambda(\bi))$, where
    \begin{equation}\label{Lamsigma'}
    \Lambda(\bi):=\left\{\sum_{k=1}^n 4^{k-1}a_k:\,a_k\in\{0,c_{i_k}\}, c_{i_k} \in C, n\ge1 \right\},
    \end{equation}
  $C=\{c_0,\ldots,c_{m-1}\}$ is a finite set of odd integers, and  $\bi=i_1i_2\cdots$ is any element in the symbol space $\Sigma_m^\infty:=\{0,\ldots,m-1\}^\infty$.  It is easy to check that $E(\Lambda(\bi))$ is an orthogonal set of $L^2(\mu_4)$ for each $\bi\in\Sigma_m^\infty$, and the discrete set $p\Lambda_1$ in \eqref{eq.intro.2} can be obtained by taking $C=\{p\}$.

  The following is our measure-theoretic characterization on the completeness of $\Lambda(\bi)$ in the sense of measure $\tilde{\mu}$.  More precisely,  $\tilde{\mu}$ is the probability measure on $\Sigma_m^\infty$ satisfying that
    \begin{equation}\label{eq.16'}
    \tilde{\mu}([i_1\cdots i_n])=m^{-n},
    \end{equation}
    where $[i_1\cdots i_n]$ is a {\it cylinder set} of $\Sigma_m^\infty$ for $m \geq 2$ (see Section \ref{eq.subsection.sym} for details).

        \begin{theo}\label{theomain1}
    With notations above, if  $m \geq 2$ and the digit set $C:=\{c_0,\ldots,c_{m-1}\}$ satisfies  $C\equiv\{1,3\}\ppmod 4$, then the set
    \begin{equation}\label{eq.intro.3}
    \cN :=\{\bi\in\Sigma_m^\infty:\, \Lambda(\bi) \text{ is not a spectrum for the measure }  \mu_4\}
    \end{equation}
   has $\tilde{\mu}$ null measure, i.e.,  $\tilde{\mu}(\cN )=0.$
        \end{theo}

    In order to keep a higher level of generality, we will obtain Theorem \ref{theomain1} as a special case of a more general result  Theorem  \ref{theo13}.  Moreover, as a refinement of Theorem  \ref{theo13}, a dichotomy law for dimensional  characterization which is formulated as  Theorem \ref{theodim-1}  (also see Theorem \ref{theodim}).

    \begin{theo}[0-1 law]\label{theodim-1}
    If  $m \geq 2$ and the digit set $C:=\{c_0,\ldots,c_{m-1}\}$ satisfies  $C\equiv\{1,3\}\ppmod 4$, then the set $$\cN(C) :=\{L \in C^\infty:\, \Lambda(L) \text{ as in \eqref{eq.3} is not a spectrum for the measure }  \mu_4\}$$
    is either an empty set or an infinite set with full Hausdorff dimension.
        \end{theo}
        Other   characterizations similar to Theorem \ref{theodim-1} is  given in Theorem \ref{theodim'}.

    This paper is organized as follows. In Section 2, we briefly recall some basic facts about $\mu_4$. In Section \ref{sec.measure},  various characterizations of  spectra for $\mu_4$ are provided, including the proofs of Theorems \ref{theomain1},  \ref{theodim-1} and \ref{coro1}.  In Section \ref{sec.eigenvalue}, we will provide an algebraic method to construct all common eigen-spectrum of $\mu_4$ corresponding to all $p\in 2\Z+1$, see Theorems \ref{lemkey} and \ref{thm.2}.

    \section{Preliminaries and basic facts}
    This section aims to briefly collect some notations and basic facts about $\mu_4$, which includes symbolic space, (quasi) 4-based expansions of integers, binary tree and labeled tree. Most of them are modified of  those in Sections 2 and 4 of \cite{DHS2009}, and they
    provide us precise notions to state and prove our main results in the sequel.
    \subsection{Symbolic space}\label{eq.subsection.sym}
    Let $m\ge2$ be an integer , we define $\Sigma_m:=\{0,1,\ldots,m-1\}$ and
       $$\Sigma_m^n:=\{i_1i_2\cdots i_n:\,i_j\in\Sigma_m,1\le j\le n\}~~~~{\rm for}~ n\ge1.$$
    When $n=0$, we set $\Sigma_m^0=\{\emptyset\}$ where $\emptyset$ is the empty word. Each element $\bi\in\Sigma_m^n$ is called a word with length $n$, whose length is also denoted by $|\bi|$, i.e., $|\bi|=n$. Denote
        $$\Sigma_m^*=\underset{0 \le  n<\infty}{\bigcup}\Sigma_m^n,$$
    and
        $$\Sigma_m^\infty:=\{i_1i_2\cdots:\,i_n\in \Sigma_m,n\ge1\}.$$
    Each element $\bi\in\Sigma_m^\infty$ is called an infinite \textit{word} or an infinite \textit{path}. In this case, the length of $\bi$ is said to be infinite, i.e., $|\bi|=\infty$.

    If $\bi\in\Sigma_m^*$ and $\bj\in\Sigma_m^*\cup\Sigma_m^\infty$, then $\bi*\bj$ denotes the concatenation of $\bi$ and $\bj$. For $n\ge1$ and $\bi=i_1\cdots i_n\in\Sigma_m^n$, define
        $$\bi^k=\mathop{\underbrace{\bi*\bi*\cdots*\bi}}\limits_{k\text{ times}}$$
    and $\bi^\infty=\bi*\bi*\cdots\in\Sigma_m^\infty$.
    For each $\bi=i_1\cdots i_{|\bi|}\in\Sigma_m^* \cup \Sigma_m^\infty$ and  $k \leq |\bi|$, the $k$th {\it node} of $\bi$ is defined by  $\bi|_k:=i_1\cdots i_k$, and we write
    $\bi^-:=i_1i_2\cdots  i_{|\bi|-1}$.

    If $\bi\in\Sigma_m^*$, the {\it cylinder set }  $[\bi]$ of the {\it symbolic space} $\Sigma_m^\infty$ is defined by
    $$[\bi]:=\{\bj\in\Sigma_m^\infty:\, \bj|_{|\bi|}=\bi\}.$$
    It is well-known that the cylinder set $[\bi]$ is both open and closed,
    and the   set $\{[\bi]: \bi \in \Sigma_m^*\}$ generates the Borel $\sigma$-algebra on $\Sigma_m^\infty$.
    It follows from  Carath\'{e}odory-Kolmogorov's extension theorem (see, e.g.,
     \cite[Theorem 1.14]{Foll1998})  that there is a unique Borel probability measure $\tilde\mu$ on the symbolic
    space  $\Sigma_m^\infty$ satisfying that
        \begin{equation}\label{eq.16}
            \tilde\mu([\bi])=m^{-|\bi|}\quad \text{for each $\bi\in\Sigma_m^*$}.
        \end{equation}
    In addition, if $I$ is a subset of $\Sigma_m^*$, we will use the symbol
        \begin{equation}\label{eq.21}
            [I]=\bigcup_{\bi\in I}[\bi]
        \end{equation}
    to  denote the union of cylinder sets in $\Sigma_m^\infty$.

    \subsection{ 4-based expansion of integers}
    For any integer $\lambda\in\Z$, there is a unique infinite word $\bi=i_1i_2\cdots\in\Sigma_4^\infty$ such that
        $$ \lambda=i_1+4i_2+4^2i_3+\cdots +4^{n-1}i_n+\cdots.$$
    As usual, $\bi$ is called the {\it  $4-$based expansion} of the integer $\lambda$.

    The following Proposition \ref{lembase4} is borrowed from \cite[Proposition 2.2]{DHS2009}.
    \begin{prop}\label{lembase4}
        An infinite word $\bi=i_1i_2\cdots\in\Sigma_4^\infty$ is the $4-$based expansion for some integer $\lambda$ if and only if $\bi$ ends in $0^\infty$ or in $3^\infty$. More precisely, if $\lambda\ge0$, then there exists a positive integer $N$ such that $i_n=0$ for all $n>N$, that is,
            $$\lambda=\sum_{n=1}^{N} 4^{n-1} i_n.$$
        If $\lambda<0$, then there exists a positive integer $N$ such that $i_n=3$ for all $n>N$, that is,
            $$\lambda=\sum_{n=1}^N4^{n-1}i_n-4^N.$$
    \end{prop}

  By Proposition \ref{lembase4}, one can easily get the following result.
    \begin{prop}\label{prop.1}
        Let $\bi=i_1i_2\cdots\in\Sigma_4^\infty$ be the $4-$based expansion of an integer $\lambda\in\Z$. Then the $n$th component $i_n$ of $\bi$ can be expressed by
        \begin{equation}\label{eqind}
            i_n=\left[\dfrac{\lambda}{4^{n-1}}\right] -4\times\left[\dfrac{\lambda}{4^n}\right], \quad   (n \ge 1),
        \end{equation}
        where $[x]$ denotes the largest integer that is no greater than $x$. Consequently, the numerator $\lambda$ in \eqref{eqind} could be replaced by any integer $\lambda'$ with $\lambda\equiv \lambda'\ppmod{4^n}$.
    \end{prop}

    \subsection{Labeled tree}\label{subsec.Labeled tree}
    In this subsection, the concepts of {\it labeled tree $(\cT, L)$} and {\it quasi  4-based expansion} with respect to the labeled tree $(\cT, L)$  are obtained by making some suitable modifications to those of Section 4 of Dutkay \textit{et al} \cite{DHS2009}. 
    \begin{defi}[\textbf{Binary tree}]\label{defn.2}
        Let $\cT$ be the \textit{complete infinite binary tree} (or {\it binary tree} for short), in which the vertex set is $\cV=\Sigma_2^*$ and the edge set is
          $\cE=\{(\bi, \bi*j):\,\bi\in\Sigma_2^*,j=0,1\}$.
        For each $\bi\in\Sigma_2^*$, we call $(\bi,\bi*j)$ the \textit{left edge} (resp. \textit{right edge}) if $j=0$ (resp. $j=1$).
    \end{defi}
    For $n\ge1$, let $E_n$ be the set of all $n$th level edges, that is,
        \begin{equation}\label{eq.15}
            E_{n}=\left\{(\bi,\bi*j):\, \bi\in\Sigma_2^{n-1},j=0,1\right\}.
        \end{equation}
    We label the edges in $E_n~ (n\ge1)$  from left to right by $a_{n,0},a_{n,1},\ldots, a_{n,2^{n}-1},$ where
    \begin{eqnarray}\label{eq.1}
    \begin{cases}
     a_{n,k_n}=0,  &     \text{if }  k_n \in \{0,2,\ldots, 2^{n}-2 \}, \\
    a_{n,k_n} \in 2\Z+1, & \text{if }   k_n \in \{ 1,3,\ldots, 2^{n}-1\}.
    \end{cases}
    \end{eqnarray}

 Next, if we read all the nonzero elements on the edges from the top to the bottom on the tree $\cT$,  from left to the right on the same level, and we will get an infinite sequence
    \begin{equation}\label{eqL}
        L= a_{1,1}a_{2,1}a_{2,3}a_{3,1}a_{3,3}a_{3,5}a_{3,7} \cdots.
    \end{equation}
  For simplicity,     we will call  $L$  a \textit{label} of the tree $\cT$ and the pair $(\cT, L)$ a \textit{labeled tree} (see Figure \ref{fig3} for the first three  levels).
    In particular, the label $L$ is said to be {\it bounded} if there is $M>0$ such that $|a_{n,i}|\leq M$ holds for all $n \in \N$ and for all odd integers $1\leq i\leq 2^n-1$.

   \newcounter{mylabel}\setcounter{mylabel}{1}
    \newcounter{mysublabel}\setcounter{mysublabel}{0}
    \newcommand{\steplab}[1][mysublabel]
        {\arabic{#1}\stepcounter{#1}}
    \newcommand{\mylabel}[3][1]{%
        \ifthenelse{#1=0}{\setcounter{mysublabel}{0}\stepcounter{mylabel}}{}
        \ifthenelse{\same=0}{
      \draw[fill=black](#2)circle[radius=.05] --node[left]{$a_{\arabic{mylabel}\steplab}$}++(#3,1) circle[radius=.05] --node[right]{$a_{\arabic{mylabel}\steplab}$}++({#3},-1) circle[radius=.05];}
        {
        \draw[fill=black](#2)circle[radius=.05] --node[left]{0}++(#3,1)circle[radius=.05] --node[right]{$c_{i_{\arabic{mylabel}}}$}++({#3},-1) circle[radius=.05];
        }
        }

    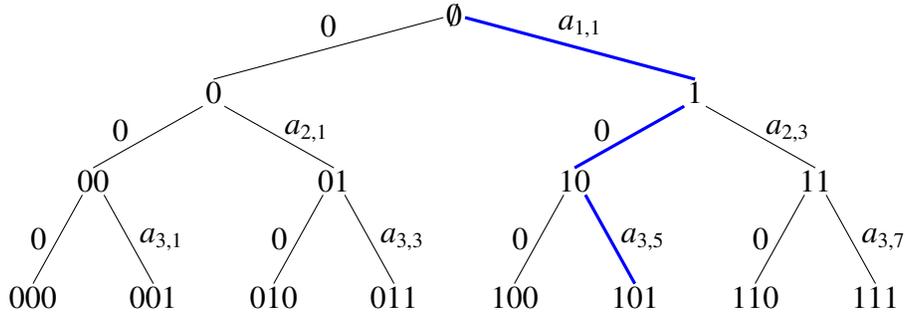
\begin{figure}[htb]
    \begin{tikzpicture}[xscale=.8]
        \pgfmathsetmacro{\h}{2}
        \pgfmathsetmacro{\y}{.18}
        \pgfmathsetmacro{\v}{-.18}
        \draw(-2*\h,-1)node{$0$}++(0,\y)--node[above]{$0$}++(2*\h-\y,1-\y) ++(\y,\y)node[yshift=-4pt]{$\emptyset$} ++(\y,-\y)--node[above]{$a_{1,1}$}++(2*\h-\y,\y-1)++(0,-\y)node{$1$};
        \foreach \i/\tex in {0/1,1/3}
            \draw(4*\i*\h-3*\h,\v-2)node{$\i0$}++(0,\y)--node[above,near start]{$0$}++(\h-\y,1-\y)++(2*\y,0)--node[above,near end]{$a_{2,\tex}$}++(\h-\y,\y-1) ++(0,-\y)node{$\i1$};
        \foreach \i/\tex in {1/{00},3/{01},5/{10},7/{11}}
            \draw(\i*\h-4.5*\h,-3+4*\v)node{$\tex0$}++(0,\y)--node[left]{$0$} ++(\h/2-\y,1-\y-2*\v)++(2*\y,0)--node[right]{$a_{3,\i}$} ++(\h/2-\y,\y-1+2*\v) ++(0,-\y)node{$\tex1$};
        \draw[blue,very thick](\y,0)--++(2*\h-\y,\y-1)++(-\y,\v-\y)-- ++(-\h+\y,\y-1)++(\y,\v-\y)--++(\h/2-\y,2*\v+\y-1);
    \end{tikzpicture}
    \caption{\label{fig3} The first three levels of the binary tree $\cT$ and   the labeled tree $(\cT, L)$.  In particular, the word $a_{1,1}a_{2,2}a_{3,5}$ represents the integer $a_{1,1}+4 \cdot 0 +4^2a_{3,5}.$}
\end{figure}

      The concept of $4-$based expansion of integers in Proposition \ref{lembase4} can be generalized as follows.
    \begin{defi}[\textbf{Quasi  $4-$based expansion}]\label{defn.1}
        Given a labeled tree $(\cT, L)$. An integer $\lambda\in\Z$ is said to have a {\it quasi $4-$based expansion} with respect to (w.r.t.) the labeled tree $(\cT,L)$ if there exists an infinite path $a_{1,k_1}a_{2,k_2}a_{3,k_3}\cdots$ such that the following two properties hold.\\ \indent
        $({\rm i})$ Each $a_{n,k_n}$ satisfies \eqref{eq.1}, and the subscript   $\{k_n\}_{n\ge1}$ satisfies that
            \begin{equation}\label{eqkn}
                0\le 2k_n\le k_{n+1}\le 2k_n+1\le 2^{n}-1, \ \mbox{ for all }n\ge1.
            \end{equation}\indent
        $({\rm ii})$ $\lambda$ satisfies the following congruence equations
            \begin{equation}\label{eqlambda}
             \lambda\equiv \sum_{i=1}^{n}4^{i-1} a_{i,k_i}\ppmod{4^n}, \ \mbox{ for all }n\ge1.
             \end{equation}
        The infinite path $a_{1,k_1}a_{2,k_2}\cdots$ is called the {\it quasi $4-$based expansion} of $\lambda$.
    \end{defi}
    A quasi $4-$based expansion $a_{1,k_1}a_{2,k_2}\cdots$ of an integer $\lambda$ is said to be \textit{finte} if there exists a positive integer $N>0$ such that $a_{n,k_n}=0$ for all $n>N$; and it is said to be {\it infinite} if
        \begin{equation}\label{eqkiinfinite}
            \#\{n:\, a_{n,k_n} \in 2\Z+1\} =+\infty.
        \end{equation}
    Here and after, $\#E$ denotes the   cardinality  of the set $E$.

    \begin{rema}\label{rem.2}
        For an integer $\lambda$, its quasi $4-$based expansion with respect to the given labeled tree $(\cT, L)$ is unique. Indeed, if $\lambda$ has two distinct quasi $4-$based expansions $a_{1,k_1}a_{2,k_2}\cdots$ and $a'_{1,k_1}a'_{2,k_2}\cdots$,  then $n:=\min\{i:\,a_{i,k_i}\ne a_{i,k_i}'\}\ge1$. This means that $a_{n,k_{n}}-a'_{n,k_{n}} \in 2\Z+1$ by Definition \ref{defn.1} (i). However, $a_{n,k_n}\equiv a'_{n,k_n}\ppmod4$ by Definition \ref{defn.1} ({\rm ii}). It is a contradiction.
        \end{rema}

    For a given  labeled tree $(\mathcal{T}, L)$, there are two discrete sets which have  the finite  and infinite quasi $4-$based expansions respectively, that is,
     \begin{gather}\label{eq.3}
     \Lambda(L)=\left\{\sum_{i=1}^{n} 4^{i-1}a_{i,k_i}:    \text{$\{a_{i,k_i}\}_{i\ge1}$  and $\{k_i\}_{i\ge1}$ satisfies \eqref{eq.1} and \eqref{eqkn} respectively}, n \in \N \right\},\\ \label{eq.7}
     \Lambda_I(L)=
            \left\{\lambda\in\Z:\,\lambda\equiv \sum_{i=1}^{n} 4^{i-1}a_{i,k_i} \!\!\!\pmod{4^n}, \text{ $\{a_{i,k_i}\}_{i\ge1}$  and $\{k_i\}_{i\ge1}$ as in \eqref{eq.1} and \eqref{eqkn}}
            \right\}.
    \end{gather}
    Obviously, $\Lambda(L)\subseteq\Lambda_I(L)$.
    In \cite[Theorem 4.4]{DHS2009}, the authors established the following fundamental theorem on the spectrality of the set $\Lambda_I(\Lambda)$ in $L^2(\mu_4)$.

    \begin{theo}[\!\cite{DHS2009}]\label{lembasic}
    Let $(\cT,L)$ be a bounded labeled tree. Then $\Lambda_I(L)$ is a spectrum for $\mu_4$.
     \end{theo}

    Theorem \ref{lembasic} immediately  yields the following corollary.
    \begin{coro}\label{rem.1}
         Let $(\cT,L)$ be a bounded labeled tree. Then  $\Lambda(L)$ is a spectrum for the measure $\mu_4$ if and only if $\Lambda(L)=\Lambda_I(L)$; equivalently, $\Lambda(L)$ is not a spectrum for $\mu_4$ if and only if there exists an integer $\lambda\in\Lambda_I(L)\setminus \Lambda(L)$ such that $\lambda$ has an infinite quasi $4-$based expansion with respect to the  labeled tree $(\mathcal{T}, L)$.
    \end{coro}

\subsection{A special labeled tree}\label{subsec.2.2}

    For the sake of further discussions about the spectral eigenvalue  problem, a special label should be paid more attention to.
    More precisely, for each level $n\ge1$, the left edges in $E_n$ as in \eqref{eq.15} are still labeled by zero but the right edges in $E_n$ are labeled by the same odd integer $a_n:=a_{n,1}=a_{n,3}=\cdots=a_{n,2^n-1}.$
    In this case, the label set $L$ given as in \eqref{eqL} will become
        \begin{equation}\label{eql2}
            L=a_1a_2a_2\cdots\underset{2^{n-1}\text{times}}{\underbrace{a_n\cdots a_n}}\cdots.
        \end{equation}
    Because the same number $a_n$ is chosen on the right edges in $E_n$, it is convenient for us to use a new symbol
        \begin{equation}\label{eq.19}
            A:=a_{1}a_{2}a_{3}\cdots,
        \end{equation}
    instead of $L$ as in \eqref{eql2}, to denote the label of the binary tree $\mathcal{T}$, see Figure \ref{fig4} for the first three  levels of the labeled tree $(\cT, A)$. The label $A$ is said to be {\it bounded} if $\{a_n\}_{n=1}^\infty$ is bounded.  {For a label $A$, denote
        \begin{gather}\label{Lamsigma}
            \Lambda(A):=\left\{\sum_{k=1}^n 4^{k-1}\omega_k:\, \omega_k\in\{0,a_{k}\},n\ge1 \right\},\\\label{eq.17}
            \Lambda_I(A):=\left\{\lambda:\,\lambda\equiv\sum_{k=1}^{n} 4^{k-1}\omega_k \ppmod{4^n} \mbox{ for all }n\ge1, \omega_k\in\{0, a_{ k}\}\right\},
        \end{gather}}

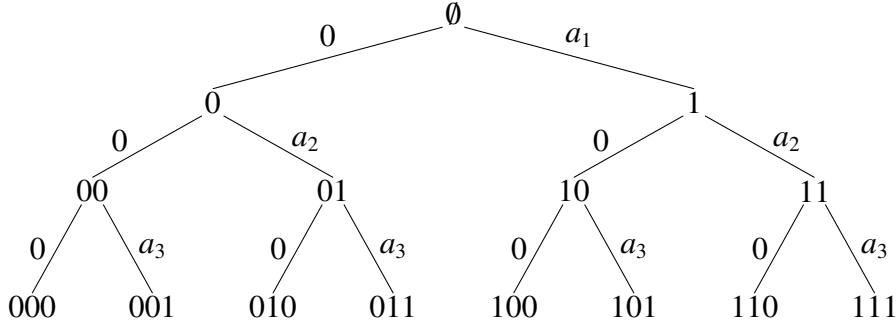
\begin{figure}[htb]
    \begin{tikzpicture}[xscale=.8]
        \pgfmathsetmacro{\h}{2}
        \pgfmathsetmacro{\y}{.18}
        \pgfmathsetmacro{\v}{-.18}
        \draw(-2*\h,-1)node{$0$}++(0,\y)--node[above]{$0$}++(2*\h-\y,1-\y) ++(\y,\y)node{$\emptyset$} ++(\y,-\y)--node[above]{$a_{1}$}++(2*\h-\y,\y-1)++(0,-\y)node{$1$};
        \foreach \i in {0,1}
            \draw(4*\i*\h-3*\h,\v-2)node{$\i0$}++(0,\y)--node[above,near start]{$0$}++(\h-\y,1-\y)++(2*\y,0)--node[above,near end]{$a_{2}$}++(\h-\y,\y-1) ++(0,-\y)node{$\i1$};
        \foreach \i/\tex in {1/{00},3/{01},5/{10},7/{11}}
            \draw(\i*\h-4.5*\h,-3+4*\v)node{$\tex0$}++(0,\y)--node[left]{$0$} ++(\h/2-\y,1-\y-2*\v)++(2*\y,0)--node[right]{$a_{3}$} ++(\h/2-\y,\y-1+2*\v) ++(0,-\y)node{$\tex1$};
    \end{tikzpicture}
    \caption{\label{fig4} The first three levels of the  binary tree $\cT$  and the special labeled tree $(\cT, A)$, where $A$ is as in \eqref{eq.19}.}
    \end{figure}

    As a consequence of  Theorem \ref{lembasic} and Corollary \ref{rem.1}, we get that the following result.

    \begin{prop}\label{prop.2}
        Assume the notations are the same as above. If the label $A$ as in \eqref{eq.19} is bounded, then $\Lambda_I(A)$ in \eqref{eq.17} is a spectrum for the measure $\mu_4$. Furthermore, $\Lambda(A)$ is a spectrum for $\mu_4$ if and only if $\Lambda(A)=\Lambda_I(A)$.
    \end{prop}
    \section{Qualitative characterization on the  spectra for middle-fourth Cantor measure}
    \label{sec.measure}

    The  purpose of this section is to 
    investigate when the set $\Lambda(L)$ defined as in \eqref{eq.3} will become a spectrum for $\mu_4$. The general measure-theoretical and dimensional characterizations are given in Theorem \ref{theo13} and Theorem \ref{theodim} (i.e., Theorem \ref{theodim-1}), respectively. As a consequence of Theorem \ref{theo13},   the proofs of Theorem \ref{theomain1} and   \ref{coro1} are given. Also, the particular variant of Theorem \ref{theo13} and Theorem \ref{theodim} will  be respectively given in Theorem \ref{theomain1'} and Theorem \ref{theodim'}. In addition, two examples (Examples \ref{exam2} and \ref{exam1}) are  constructed to illustrate the theory of Theorem \ref{theo13}.


    We begin with the concept of {\it incomplete sets} in symbolic space. This is the starting point of our research on the characterizations of spectra for $\mu_4$.
    \subsection{Incomplete set}
      Let $m\geq2$ be an integer, and let $\Sigma_m^*$ and $\Sigma_m^\infty$ be given as in Section 2.  The notation of \textit{incomplete set} of $\Sigma_m^\infty$ is defined as follows.
    \begin{defi}\label{defn.incom}
     Suppose that $\Theta$ is a subset of $\Sigma_m^\infty$.\\
     \indent$({\rm i})$  Given an $\bi\in\Theta$, we say that  $\bj\in\Sigma_m^*$ is an \textbf{incomplete node} of $\bi$ if ~$\bj=\bi|_k$ for some $k \in \N$ and there exists a  $j\in \Sigma_m$ such that  $[\bj*j]\cap\Theta=\emptyset$.\\
     \indent$({\rm ii})$  An $\bi\in\Theta$ is said to be an \textbf{incomplete path}
       if $\bi$ has infinitely many  incomplete nodes. \\
       \indent$({\rm iii})$  The set  $\Theta$ is called an \textbf{incomplete set} of $\Sigma_m^\infty$ if each $\bi\in\Theta$ is an incomplete path.
        \end{defi}
   Let $\tilde\mu$ be a Borel probability measure   on  $\Sigma_m^\infty$ as in \eqref{eq.16}. Lemma \ref{lempoor}  shows that  the incomplete set  of $\Sigma_m^\infty$ is of   $\tilde\mu$-null set, which plays  an important role in the proof of Theorem \ref{theo13}.

    \begin{lem}\label{lempoor}
        If $\Theta \subseteq \Sigma_m^\infty$ is  an incomplete set of $\Sigma_m^\infty $, then   $\tilde\mu(\Theta)=0$.
    \end{lem}
    \begin{proof}
        For each $\bi\in\Theta$ and each $k\ge1$, we use $N_k(\bi)$ to denote the length of $k$th incomplete node of $\bi\in\Theta$. Let (see Figure \ref{fig1} for an illustration of $m=4$)
        \begin{gather}\label{eqI_k}
            I_k=\{\bi|_{N_k(\bi)}:\,\bi\in\Theta\},
        \end{gather}
        and
        \begin{gather}\label{eqJ_k}
            J_k=\{\bj*j:\,[\bj*j]\cap \Theta\ne\emptyset, \bj\in I_k, j \in \Sigma_m\}.
        \end{gather}

    \newcommand{\mypara}[4][0]{%
    \ifthenelse{#1=0}{
    \draw #2sin++(.1,-.1)--++({(#3)/2-.22},0)--++(.12,-.15) node[yshift=-5pt]{#4}--++(.12,.15)--++({(#3)/2-.22},0) cos++(.1,.1);}
    {\draw #2cos++(.1,.1)--++({(#3)/2-.22},0)--++(.12,.15) node[yshift=5pt]{#4}--++(.12,-.15)--++({(#3)/2-.22},0) sin++(.1,-.1);}}

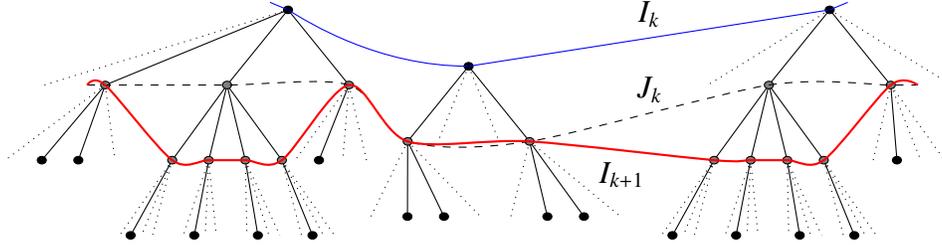
\begin{figure}[htb]
\begin{tikzpicture}[xscale=1.2]
    \pgfmathsetmacro{\r}{.05}
    \pgfmathsetmacro{\hy}{2}   
    \pgfmathsetmacro{\y}{-.75} 
    \pgfmathsetmacro{\h}{1.35} 
    \pgfmathsetmacro{\hh}{.15*\h}
    \pgfmathsetmacro{\hhh}{.3*\hh}
    \pgfmathsetmacro{\v}{1}    
    \coordinate (a1) at (0,0);
    \coordinate (a2) at (\hy,\y);
    \coordinate (a3) at (3*\hy,0);
    \coordinate (b1) at (-1.5*\h,-\v);
    \coordinate (b2) at ($(b1)+(\h,0)$);
    \coordinate (b3) at ($(b2)+(\h,0)$);
    \coordinate (b4) at ($(a2)+(-.5*\h,-\v)$);
    \coordinate (b5) at ($(a2)+(.5*\h,-\v)$);
    \coordinate (b6) at ($(a3)+(-.5*\h,-\v)$);
    \coordinate (b7) at ($(a3)+(.5*\h,-\v)$);
    \coordinate (c1) at ($(b2)+(-3*\hh,-\v)$);
    \coordinate (c2) at ($(b2)+(-1*\hh,-\v)$);
    \coordinate (c3) at ($(b2)+(1*\hh,-\v)$);
    \coordinate (c4) at ($(b2)+(3*\hh,-\v)$);
    \coordinate (c5) at ($(b6)+(-3*\hh,-\v)$);
    \coordinate (c6) at ($(b6)+(-1*\hh,-\v)$);
    \coordinate (c7) at ($(b6)+(1*\hh,-\v)$);
    \coordinate (c8) at ($(b6)+(3*\hh,-\v)$);
    \foreach \i in {a1,a2,a3}
        \draw[fill=black](\i)circle(\r);
    \draw[blue]($(a1)+(-.2,.1)$)--(a1)parabola bend (a2) (a2)--(a3)--++(.2,.1);
    \foreach \i/\sty in {-3/black,-1/black,1/black,-4/dotted}
        \draw[\sty](a1)--++(\i*\h*.5,-\v);
    \foreach \i/\sty in {.5/dotted,-1/black,1/black,-.5/dotted}
        \draw[\sty](a2)--++(\i*\h*.5,-\v);
    \foreach \i/\sty in {-2/dotted,-1/black,1/black,2/dotted}
        \draw[\sty](a3)--++(\i*\h*.5,-\v);
    \foreach \a in {b2,b6}
        \foreach \i/\yy in {-3/-.4,-1/-.1,1/.2,3/.4}
            {\draw(\a)--++(\i*\hh,-\v);
            \foreach \j/\sty in {-3/dotted,-1/black,1/dotted,3/dotted}
                {\draw[\sty]($(\a)+(\i*\hh,-\v)$)--++(\j*\hhh+\yy,-\v);
                \ifthenelse{\j=-1}
                    {\draw[fill=black]($(\a)+(\i*\hh+\j*\hhh+\yy,-2*\v)$) circle(\r);}{;}
                }}
    \foreach \a/\yy in {b1/-.5,b4/.2,b5/.4}
        \foreach \i/\sty in {-3/dotted,-1/black,1/black,3/dotted}
            {\draw[\sty](\a)--++(\i*\hh+\yy,-\v);
            \ifthenelse{\i=-1}
                {\draw[fill=black]($(\a)+(\i*\hh+\yy,-\v)$)circle(\r);}{
                \ifthenelse{\i=1}{\draw[fill=black]($(\a)+(\i*\hh+\yy,-\v)$)circle(\r);}
                {;}}
            }
    \foreach \a/\yy in {b3/-.2,b7/.2}
        \foreach \i/\sty in {-3/dotted,-1/black,1/dotted,3/dotted}
            {\draw[\sty](\a)--++(\i*\hh*.65+\yy,-\v);
            \ifthenelse{\i=-1}
            {\draw[fill=black]($(\a)+(\i*\hh*.65+\yy,-\v)$)circle(\r);}
            {;}}
    \foreach \i in {b1,b2,b3,b4,b5,b6,b7,c1,c2,c3,c4,c5,c6,c7,c8}
        \draw[fill=gray](\i)circle(\r);
    \draw[dashed,smooth]plot coordinates{($(b1)+(-.2,0)$)(b1)(b2)(b3)(b4)(b5)(b6) (b7)($(b7)+(.3,0)$)};
    \draw[thick,red,smooth]plot coordinates {($(b1)+(-.2,0)$)(b1)(c1)(c2)(c3)(c4)(b3)(b4)(b5)(c5)(c6)(c7)(c8) (b7)($(b7)+(.3,0)$)};
    \foreach \i/\j/\tex/\sty in {a2/a3/I_k/above,b5/b6/J_k/above,b5/c5/I_{k+1}/below}
        \draw($.5*(\i)+.5*(\j)$)node[\sty]{$\tex$};
\end{tikzpicture}
    \caption{\label{fig1}
    An illustration for $I_k,J_k$ in $\Sigma_4^\infty$. The points lying in the first horizontal  solid  line form the set $I_k$, and the points lying in the second horizontal dashed line form the set $J_k$, and the points lying in the third curve form the set $I_{k+1}$. }
     \end{figure}

    As in \eqref{eq.21}, the sets $[I_k]$ and $[J_k]$ are the union of cylinder sets.
    It is not hard to see that $\{[I_k]\}_{k=1}^\infty$  is a  sequence of decreasing sets of $\Sigma_m^\infty$ and $\Theta\subseteq\bigcap_{k\ge1} [I_k]$.
    This yields that
    $$\tilde{\mu}(\Theta) \leq\tilde{\mu}([I_k])$$
    holds for all $k \in \N$. Thus, the   task left for us is to estimate the size of each $\tilde{\mu}([I_k])$.

    \smallskip
    In fact, since each $\bj\in I_k$  is an  incomplete node of some $\bi \in \Theta$, it follows that the cardinality of the set
    $$A_\bj:=\{j\in \Sigma_m:\,[\bj*j]\cap \Theta\ne\emptyset\}$$
    is less than $m$, which implies that
    $$\tilde\mu\bigg(\bigcup_{j\in A_\bj}[\bj*j]\bigg)
    =\sum_{j\in A_\bj}\tilde\mu([\bj*j])
    =\dfrac1{m}\sum_{j\in A_\bj}\tilde\mu([\bj])
    \le  \dfrac{m-1}{m}\tilde\mu([\bj]).$$
    Therefore,  according to the definitions of $[I_{k}]$ and $[J_{k}],$ one has
    $$\tilde\mu([J_{k}])
    =\tilde\mu\bigg(\bigcup_{\bj\in I_k}\bigcup_{j\in A_\bj}[\bj*j]\bigg)
    =\sum_{\bj\in I_k}\tilde\mu\bigg(\bigcup_{j\in A_\bj}[\bj*j]\bigg)
    \le \dfrac{m-1}{m} \tilde\mu([I_k]).$$
    By the definitions of $I_k, J_k$, one has that the  inclusion relation $[I_{k+1}]\subseteq [J_k]$ holds for all $k \in \N$. Whence,
    $$\tilde\mu([I_{k+1}])\le \dfrac{m-1}{m}\tilde\mu([I_k]) \qquad (\forall \ k \in \N).$$
    Combining this with the fact that $[I_1]\subset\Sigma_m^\infty$ and $\tilde\mu(\Sigma_m^\infty)=1$,    one gets that
    $$\tilde\mu([I_k])\le \left(\dfrac{m-1}{m}\right)^{k-1}, \qquad (\forall \ k \in \N).$$
    Letting $k \rightarrow \infty$, we  conclude that $\tilde\mu(\Theta)=0$.
    The proof of Lemma \ref{lempoor} is completed.
        \end{proof}

    \subsection{Characterizations of the spectra from the viewpoint of Measure and Dimension}
    Here and below, we always assume that $C:=\{c_0,\ldots,c_{m-1}\}\subseteq 2\Z+1,$ where $2\le m\in\N$. For simplicity of notations, we always use the symbol $\cL$ to denote the set of infinite words generated by $C$, that is,
        \begin{equation}\label{eq.5}
        \cL:=C^\infty=C \times C \times  \cdots =\{c^{(1)}c^{(2)}\cdots:\,c^{(j)}\in C\}.
        \end{equation}

    We next define the {\it image measure} $\mu$ on $\mathcal{L}$ as follows,
        \begin{equation}\label{eqmu}
            \mu(\cdot):=\tilde\mu\circ \pi^{-1}(\cdot),
        \end{equation}
    where $\pi\,:\Sigma_m^\infty\mapsto\cL$ is the natural projection mapping
         \begin{equation}\label{eq.18}
            \pi(i_1i_2i_3\cdots)= c_{i_1}c_{i_2}c_{i_3}\cdots,
        \end{equation}
     and $\tilde\mu$ is the Borel probability measure on $\Sigma_m^\infty$  defined by \eqref{eq.16}. Clearly, $\pi$ is a bijection.

    Our first main result Theorem \ref{theo13} below shows that for $\mu$-almost all $L \in \cL$,     the set  $\Lambda(L)$ is a spectrum for $\mu_4$, which refines \cite[Theorem 4.4]{DHS2009} (see Theorem \ref{lembasic})  from the viewpoint of measure.
    \begin{theo}\label{theo13}
        Let $\mu$ defined as in \eqref{eqmu} be the Borel probability measure  on $\cL$ as in \eqref{eq.5}. If the finite digit set $C:=\{c_0,\ldots,c_{m-1}\}\subseteq 2\Z+1$ satisfies that
        \begin{equation}\label{eqC13}
            C\equiv\{1,3\}\ppmod4,
        \end{equation}
    then for $\mu$-a.e.  $L\in \cL$, the set $\Lambda(L)$ as in \eqref{eq.3} forms a spectrum for $\mu_4$.
    \end{theo}
    \begin{proof}
        We will use Corollary \ref{rem.1} to show that $\mu\big(\cN(C)\big)=0$ if \eqref{eqC13} holds, where
        \begin{equation}\label{eq.6}
            \cN(C):=\{L \in \cL: \text{$\Lambda(L)$  as in \eqref{eq.3} is not a spectrum for $\mu_4$}\}.
        \end{equation}
        For this purpose, if $\lambda\in\Z$, we set
        \begin{equation}\label{eq.22}
            \cN_\lambda(C)=  \left \{L \in \cL: \text{$\lambda\in \Z$ has an infinite quasi $4-$based expansion w.r.t. $(\mathcal{T}, L)$}\right\}.
        \end{equation}
        By Corollary \ref{rem.1}, we know that $\cN(C)= \bigcup_{\lambda\in\Z}\cN_\lambda(C)$. Thus, it suffices to show that $\mu\big(\cN_\lambda(C)\big)=0$ holds for all $\lambda\in\Z$. Next, we fix $\lambda\in\Z$ and to show $\mu\big(\cN_\lambda(C)\big)=0$ in the following two cases.

        (1) If $\cN_\lambda(C)$ is an empty set (it is possible, for instance, see Theorem \ref{theodim} (i) below), it is clear that $\mu(\cN_\lambda(C))=0$.

        (2) If $\cN_\lambda(C)$ is not an empty set, we will get the desired result if one can show that $\pi^{-1}\big(\cN_\lambda(C)\big)$ is an incomplete set of $\Sigma_m^\infty$, where $\pi$ is defined as in \eqref{eq.18}. This is true because
            $$\mu\big(\cN_\lambda(C)\big)=\tilde{\mu}\Big(\pi^{-1} \big(\cN_\lambda(C)\big)\Big)=0,$$
        by \eqref{eqmu} and Lemma \ref{lempoor}.

        It remains to show that $\pi^{-1}\big(\cN_\lambda(C)\big)$ is an incomplete set of $\Sigma_m^\infty$. In fact, if $\bi\in\pi^{-1}\big(\cN_\lambda(C)\big),$ there is a  \textbf{unique} label $L \in \cN_\lambda(C)$ such that $\bi=\pi^{-1}(L)$. Without loss of generality, we may assume the label $L$ as in \eqref{eqL}:
            $$L= a_{1,1}a_{2,1}a_{2,3}a_{3,1}a_{3,3}a_{3,5}a_{3,7}\cdots \in  \cL.$$
        According to Definition \ref{defn.1} and Remark \ref{rem.2}, the fixed integer $\lambda$  has the unique infinite expansion $a_{1,k_1}a_{2,k_2}\cdots$ with respect to the labeled tree $(\cT,L)$  such that \eqref{eqkn}, \eqref{eqlambda} and \eqref{eqkiinfinite} hold. Now \eqref{eqkiinfinite} implies that the cardinality of the set
            \begin{equation}\label{eq.10-1}
                I=\{n:\,a_{n,k_n}\in C\}
            \end{equation}
        is infinity.

        Noting that for each $n\in I$ the odd integer $a_{n,k_n}$ for $k_n\in \{1,3,\ldots,2^{n}-1\}$, is located at the $l_n$th position of the label $L$, where
            $$l_n:=2^{n-1}+\bigg[\frac{k_n-1}2\bigg],$$
        and the assumption \eqref{eqC13} implies that for each odd integer $a\in C$, there is some $c\in C$ such that $a+c\equiv0\ppmod4$ holds.    Therefore, if the $l_n$th number $a_{n,k_n}$ of the label $L$ is replaced by other $c_j\in C$ with $c_j+a_{n,k_n}\equiv0\ppmod4$, then \eqref{eqlambda} fails, i.e.,
            $$\lambda\not\equiv\sum_{i=1}^{n-1}4^{i-1}a_{i,k_i} +4^{n-1}c_j\pmod{4^{n}}.$$
        Combining with the fact that $\pi:\Sigma_m^\infty\mapsto\cL$ as in \eqref{eq.18} is a bijection, the definition of $\cN_\lambda(C)$ in \eqref{eq.22} implies  that
            $$[\bi|_{l_n-1}*j]\cap \pi^{-1}\big(\cN_\lambda(C)\big)=\emptyset.$$
        The above arguments tell us that  finite words $\{\bi|_{l_n-1}\}_{n\in I}$  are all incomplete nodes of $\bi$,  and hence   $\bi$ is an infinite incomplete path in $\Sigma_m^\infty$ since $I$ is an infinite set by \eqref{eq.10-1}. Therefore, $\pi^{-1}(\cN_\lambda(C))$ is an incomplete set of $\Sigma_m^\infty$ since $\bi\in\pi^{-1}\big(\cN_\lambda(C)\big)$ is arbitrary. This ends the proof of Theorem \ref{theo13}.
    \end{proof}

    \begin{rema}
        In terms of Theorem \ref{lembasic},  the set $\cN(C)$ in \eqref{eq.6} can be written as
            \begin{equation}\label{eq.8}
                \mathcal{N}(C)=\{L \in \cL:  \Lambda(L) \not =\Lambda_I(L), \ L \in \cL=C^\infty\},
            \end{equation}
        where $\Lambda(L)$ and $\Lambda_I(L)$ are defined by \eqref{eq.3} and \eqref{eq.7}, respectively.
        Thus, Theorem \ref{theo13} can be restated as:
        If $C\subseteq 2\Z+1$ is a finite set satisfying that $C\equiv\{1,3\}\ppmod4$, then $\mu(\mathcal{N}(C))=0$.
    \end{rema}

    As the set $\cN(C)$ (see \eqref{eq.6} or \eqref{eq.8}) is always a $\mu$-null set for any  finite digit set $C\subseteq2\Z+1$, it is naturally hopeful to know how large the cardinality of the set  $\cN(C)$ is.
    Our second main result Theorem \ref{theodim} refines Theorem \ref{theo13} from the viewpoint of dimension.
    \begin{theo}[0-1 law]\label{theodim}
    Assume that  $C:=\{c_0,\cdots,c_{m-1}\}\equiv\{1,3\}\ppmod4$,    we write $\cN(C)$ as
    $$
    \cN(C):=\{L \in \cL: \text{$\Lambda(L)$  as in \eqref{eq.3} is not a spectrum for $\mu_4$}\}.
    $$
    Then either
    $({\rm i})$ $\cN(C)$ is an empty set,  or $({\rm ii})$ $\cN(C)$ is an infinite set with full Hausdorff dimension.
    \end{theo}
    \begin{proof}
    For $({\rm i})$, we set $C= \{1,7\}$. Suppose on the contrary that $\cN(C)$ is not an empty set. Then there exists a label $L \in \cL$ and a nonzero integer $\lambda$ such that it has the following  infinite quasi $4-$based expansion (with respect to $(\mathcal{T}, L)$)
    \begin{equation}\label{eq.exam.1}
     \lambda \equiv \sum_{i=1}^{n} 4^{i-1} a_{i} \!\!\!\pmod{4^n}, \qquad \text{ for all $n \geq 1$},
    \end{equation}
    where $a_1a_2a_3\cdots \in \{0,1,7\}^\infty$ and $I=\{n: a_{n} \in C\}$ is an infinite set. We will obtain $(i)$ by the following two claims.

        \noindent\textbf{Claim 1.} The segments $01, 10, 11, 07, 70, 77$  do not  appear infinitely often in  the quasi $4-$based expansion expansion $a_1a_2a_3\cdots$ as in \eqref{eq.exam.1}.

    \noindent\textbf{Claim 2.} The segment $171$ does not appear infinitely often in the quasi $4-$based expansion $a_1a_2a_3\cdots$ as in \eqref{eq.exam.1}.

    In fact,  if Claim 1 and Claim 2 hold, then for any label $L \in \cL$, there is no integer $\lambda\in \Lambda_I(L) \setminus \Lambda(L)$ such that \eqref{eq.exam.1} holds in which $I=\{n: a_{n} \in C\}$ is an infinite set. This  means that $\cN(C)$ is an empty set, so (i) follows.

  Next, we show Claims 1 and 2.

\smallskip
  \smallskip\noindent\textit{Proof of Claim 1.} Suppose   on the contrary that there are infinitely many $n \in \N$ such that $a_{n+1}a_{n+2}=77.$  It then yields contradictions as follows.
    First, since
    \begin{equation}\label{eqest51}
    i_{n}:=\left[\dfrac{\sum_{i=1}^n4^{i-1}a_i}{4^{n-1}}\right]
    <\left[\dfrac{7\sum_{i=1}^n4^{i-1}}{4^{n-1}}\right]<3, \qquad (\forall n\in \N),
    \end{equation}
    we get that the $4-$based expansion of $\sum_{i=1}^n4^{i-1}a_i$ is $i_1i_2\cdots i_{n-1}i_n\in \Sigma_4^n$ with $i_n\in\{0,1,2\}$, namely
    \begin{equation}\label{eq.9}
    \sum_{i=1}^n4^{i-1}a_i = \sum_{j=1}^{n}4^{j-1}i_j,\qquad \text{where} \qquad i_n\in\{0,1,2\}.
    \end{equation}
    Second, the segment $77$ is the quasi $4-$based expansion of the integer  $35$, which has the $4-$based expansion $302$.
    Thus, if $a_{n+1}a_{n+2}=77$, we can easily obtain that the $4-$based expansion of $\sum_{i=1}^{n+2}4^{i-1}a_i$ is $i_1\cdots i_{n-1} i_n'i_{n+1}'i_{n+2}'\in\Sigma^{n+2}_4$ with $i'_n i'_{n+1}i'_{n+2}\in \{302,012,112\}$. By \eqref{eq.exam.1} and Proposition \ref{prop.1}, if there are infinitely many $n$ such that $a_{n+1}a_{n+2}=77$, then the integer $\lambda$ does not end in $0^\infty$ or $3^\infty$. It is a contradiction to Proposition \ref{lembase4}.

    Similarly, it is not hard to see that the segment $07$ (resp. $70$) is the quasi $4-$based expansion of the integer $28$, which has the $4-$based expansion $031$ (resp. $31$). Thus, by doing the same arguments as we did above for $a_{n+1}a_{n+2}=77$,  we will get that  Claim 1 is true for the segments $01, 10, 11, 07, 70$. This completes the proof of Claim 1.

     \smallskip\noindent\textit{Proof of Claim 2.} We assume on the contrary that there are infinitely many $n \in \N$ such that
    $a_{n+1}a_{n+2}a_{n+3}=171.$ In this case, we know that the quasi $4-$based expansion $171$ and the $4-$based expansion $132$ represent the same positive integer $45$. Thus, by doing the same procedure as we did for the segment $77$ in the proof of Claim 1, one can obtain, from \eqref{eqest51} and \eqref{eq.9}, that if $a_{n+1}a_{n+2}a_{n+3}=171$, then the $4-$based expansion of $\sum_{i=1}^{n+3}4^{i-1}a_i$ is $i_1i_2\cdots i_{n-1}i'_n i'_{n+1}i'_{n+2}$, where $i'_n i'_{n+1}i'_{n+2}=132,232$ or $332$.  Thus, by \eqref{eq.exam.1} and Proposition \ref{prop.1} again, the integer $\lambda$ does not end in $0^\infty$ or $3^\infty$.  It is also a contradiction to Proposition \ref{lembase4}. This completes the proof of Claim 2.

\smallskip

     $({\rm ii})$ \ If $\cN(C)$ is not an empty set, by Theorem \ref{lembasic} and \eqref{eq.8}, for any label $$L:= a_{1,1}a_{2,1}a_{2,3}a_{3,1}a_{3,3}a_{3,5}a_{3,7}\cdots \in  \cN(C),$$ there is   an integer  $\lambda\in \Lambda_I(L)\setminus \Lambda(L)$ such that $\lambda$ has the following unique infinite quasi $4-$based expansion  with respect to the labeled tree $(\mathcal{T}, L)$,
     \begin{equation}\label{eq.10}
      \lambda\equiv \sum_{i=1}^{n}4^{i-1} a_{i,k_i}\ppmod{4^n},\qquad (n\ge1),
     \end{equation}
     where $a_{i,k_i}$  and $k_i$ are respectively satisfying Definition \ref{defn.1} ({\rm i}) and ({\rm ii}), and the set
     \begin{equation*}
     I=\{n:\,a_{n,k_n}\in C\},
     \end{equation*}
     has infinity many elements.  We consider the set
     \begin{equation*}
    \widetilde{\cN}_\lambda(C)=\{L= b_{1,1}b_{2,1}b_{2,3}b_{3,1}b_{3,3}b_{3,5}b_{3,7} \cdots: b_{n,k_n}=a_{n,k_n} \text{ for all $n \in I$}\},
     \end{equation*}
    which  clearly has Hausdorff dimension $1$ (e.g., cf. \cite{F1}).

    On the other hand, it is obvious that,
     for each label $L\in\widetilde{\cN}_\lambda(C)$, the integer $\lambda$ above also satisfies  the congruence equations \eqref{eq.10},
     which gives the unique infinite quasi base expansion  with respect to the labelled tree  $(\mathcal{T},L)$ (here, the uniqueness follows from Remark \ref{rem.2}).
     Consequently,  $\widetilde{\cN}_\lambda(C)$ is a subset of $ {\cN}(C)$. Thus the set $ {\cN}(C)$ also has Hausdorff dimension 1.
      \end{proof}

    We end this subsection by giving two examples to 
    analyze the condition   \eqref{eqC13} of Theorem \ref{theo13}. To be precise, Example \ref{exam2}   shows that  \eqref{eqC13}  can not be removed  in Theorem \ref{theo13}, while Example \ref{exam1} implies that \eqref{eqC13} is sufficient but not necessary. In fact, Example \ref{exam1} also provides us new spectra of $\mu_4$ which are never found in the previous literatures, e.g., see \cite{DHS2009,Dai2016,F2019}.

    \begin{exam}\label{exam2}
    Let $C=\{3,15\}$. Then \eqref{eqC13} fails,  and for each label $L\in\cL$, the set $\Lambda(L)$ is not a spectrum for the measure $\mu_4$.
    \end{exam}
    \begin{proof}
        The condition \eqref{eqC13} clearly fails since $3\equiv15 ~\ppmod 4$.
        For any label $L\in \cL$, we consider the  infinite path $a_{1,k_1}a_{2,k_2}\cdots \in \{0,3,15\}^\infty$  on the labeled tree $(\mathcal{T}, L)$ which satisfies   the following conditions:
        $k_1=1$, $a_{1,k_1}=3 \ \text{or} \ 15$, and  for $n\ge1$,
        \begin{eqnarray*}
        a_{n+1,k_{n+1}}=\begin{cases}
                    0,& \mbox{if }a_{n,k_n}=15,\\
                    3 \ \text{or } 15 ,&\mbox{if }a_{n,k_n}=0 \ \text{or }  3,
                    \end{cases} \quad \hbox{where} \quad
        k_{n+1}=\begin{cases}
        2k_n,& \mbox{if }a_{n,k_n}=15,\\
        2k_n+1,& \mbox{if }a_{n,k_n}=0 \ \text{or }  3.
                    \end{cases}
        \quad
        \end{eqnarray*}
        It is not hard to see
        that the set $I:=\big\{n:\,a_{n,k_n}\in \{3, 15\}\big\}$ is an infinite set, and one can check that
        \begin{equation*}
        -1\equiv \sum_{i=1}^{n}4^{i-1} a_{i,k_i}\ppmod{4^n},\qquad \text{for all $n\in \N$},
        \end{equation*}
        which is equivalent to that $a_{1,k_1}a_{2,k_2}\cdots \in \{0,3,15\}^\infty$
        is the infinite quasi $4-$based expansion of $-1$.
        This yields that $-1 \in \Lambda_I(L) \setminus \Lambda(L)$, and thus  $\Lambda(L)$ is not a spectrum for the measure $\mu_4$ by  Corollary  \ref{rem.1}.
        The proof of Example \ref{exam2} is completed.
         \end{proof}

        \begin{exam}\label{exam1}
        Let $C=\{1,5\}$. Then for each label $L\in\cL$, the set $\Lambda(L)$ is a spectrum for $\mu_4$.
        \end{exam}
        \begin{proof}
        We suppose that there is a label $L\in\cL$ such that  $\Lambda(L)$ is not a spectrum for $\mu_4$.
        Then Theorem \ref{lembasic} implies that there exists an integer $\lambda \in \Lambda_I(L) \setminus \Lambda(L)$ such that $\lambda$ has an infinite quasi $4-$based expansion $a_{1,k_1}a_{2,k_2}\cdots \in \{0,1,5\}^\infty$, that is,
        \begin{equation}\label{eq.4}
        \lambda\equiv \sum_{i=1}^{n}4^{i-1} a_{i,k_i}\ \ppmod{4^n}, \quad ( n\ge1),
        \end{equation}
        and the set $I:=\big\{n:\,a_{n,k_n}\in \{1,5\}\big\}$ is an infinite set.

        Assume that $i_1i_2i_3\cdots$ is the $4-$based expansion of $\lambda$. By  Proposition \ref{prop.1} and \eqref{eq.4}, we get that the $n$th component $i_n$ of $\lambda$ is
        \begin{equation*}
        i_n = \bigg[\dfrac{\sum_{i=1}^n4^{i-1}a_{i,k_i}}{4^{n-1}}\bigg]-4 \bigg[\dfrac{\sum_{i=1}^n4^{i-1}a_{i,k_i}}{4^{n}}\bigg].
        \end{equation*}
        Note  that for all $n > 1,$ one has
         $$0\le \sum_{i=1}^{n-1}4^{i-1}a_{i,k_i}\le \sum_{i=1}^{n-1 }4^{i-1}\times5< \frac{5}{3} \times 4^{n-1 }.$$
        Therefore, one can easily obtain that
        $$i_n=\begin{cases}
              1, & \text{if } a_{n, k_n}=1; \\
              1\text{ or }2, &\text{if } a_{n, k_n}=1.
           \end{cases}$$
        This means that the $4-$based expansion $i_1i_2i_3\cdots$ of $\lambda$ does not end in $0^\infty$ or $3^\infty$, contradicting to Proposition \ref{lembase4}.
        Thus, the above arguments show that for each label $L\in\cL$ the set $\Lambda(L)$ is a spectrum for $\mu_4$. The proof of Example \ref{exam1} is completed.
        \end{proof}

    \subsection{Particular cases of Theorem \ref{theo13}} 
    Let us now concentrate on the application of Theorem \ref{theo13} to the special labeled tree $(\cT, A)$,  which are described in Subsection \ref{subsec.2.2}. As before, for any integer $m\ge2$ and the  finite digit set $C=\{c_0,\ldots,c_{m-1}\}\subseteq 2\Z+1$, we still define the {\it image measure} $\mu$ on $\mathcal{L}$ as in \eqref{eqmu},  and let
    \begin{equation}\label{eq.6'}
    \mathcal{N}(C)=\{A \in \mathcal{L}:  \Lambda(A) \not =\Lambda_I(A)\},
    \end{equation}
    where $\Lambda(A)$ and $\Lambda_I(A)$ are respectively defined by \eqref{Lamsigma} and \eqref{eq.17}.

    The following result immediately follows from Theorem \ref{theo13}. We omit the details of the proof.
    \begin{theo}\label{theomain1'}
    Let $\mu$  defined as in \eqref{eqmu} be the Borel probability measure  on $\cL$ as in \eqref{eq.5}.
    If the finite digit set $C:=\{c_0,\ldots,c_{m-1}\}$ satisfies that $C\equiv\{1,3\}~\ppmod4$, then for $\mu$-a.e.  $A\in \cL$, the set $\Lambda(A)$ as in \eqref{Lamsigma} forms a spectrum for $\mu_4$. That is, $\mu\big(\cN(C)\big)=0$, where $\cN(C)$ is defined as in \eqref{eq.6'}.
    \end{theo}
    Now one can give the proofs of Theorems \ref{theomain1}  and   \ref{coro1}.

    \smallskip
    \begin{proof}[Proof of Theorem \ref{theomain1}]
    Let $\Lambda(\bi)$ and $\cN$ be the sets which are respectively defined as in \eqref{Lamsigma'} and \eqref{eq.intro.3}.
    Since the mapping  $\pi: \Sigma_m^\infty \rightarrow \cL$ is a bijection, it follows that for any $A=a_1a_2\cdots \in \cL$, there exists a  unique $\bi=i_1i_2\cdots$ such that $a_j=c_{i_j}$ for all $j \in \N$, that is,
    $$A=\pi(\bi)=c_{i_1}c_{i_2}\cdots,$$
    which is equivalent to say that $\Lambda(\bi)=\Lambda(A),$ where $\Lambda(A)$ is written as in \eqref{Lamsigma}, and hence $\cN=\pi^{-1}(\cN(C)).$ Therefore,
        $$\tilde{\mu}(\cN )=\tilde{\mu}\big(\pi^{-1}\cN(C) \big)={\mu}\big(\cN(C)\big)=0. $$
    This finishes the proof  of Theorem \ref{theomain1}.
    \end{proof}

     \begin{rema}
    It is not hard to see that the statements of Theorem \ref{theomain1'} and Theorem   \ref{theomain1} are equivalent.
     \end{rema}

    As a consequence of Theorem \ref{theomain1'} or Theorem   \ref{theomain1}, we obtain the proof of Theorem \ref{coro1}.

    \smallskip
    \begin{proof}[Proof of Theorem \ref{coro1}]

   For $p\in 2\Z+1$, we set $C=\{1,3\}$, $C_p:=pC=\{p,3p\}$.
    It is clear that $$C_p\equiv\{1,3\}\ppmod{4}.$$
    As in \eqref{eqmu} and \eqref{eq.18}, we define  an image measure $\mu_p$ on $C_p^\infty$ as $\mu_p(\cdot):=\tilde\mu\circ \pi_p^{-1}(\cdot)$, and define the natural projection  mapping $\pi_p: \Sigma_2^\infty=\{0,1\} ^\infty \rightarrow C_p^\infty$ as
    \begin{equation*}
    \pi_p(i_1i_2i_3\cdots)= c_{i_1}c_{i_2}c_{i_3}\cdots,
    \end{equation*}
    where $c_{i_j} \in C_p$ for all $j \in \N$.
  Then the set
    \begin{equation*}
    \cN(C_p)=\{A \in C_p^\infty: \text{$\Lambda(A)$ as in \eqref{Lamsigma}   is not a spectrum for $\mu_4$}\}
    \end{equation*}
    satisfies that $\pi_p^{-1}\big(\cN(C_p)\big) \subseteq \Sigma_2^\infty $ and $\tilde{\mu}\Big(\pi_p^{-1}\big(\cN(C_p)\big)\Big)=\mu\big(\cN(C_p)\big)=0$ by Theorem \ref{theomain1'}.

    It follows that the set
    $$\Xi=\bigcup_{p\in 2\Z+1} \pi_p^{-1}\big(\cN(C_p)\big)$$
    satisfies that $\Xi \subseteq \Sigma_2^\infty$ and $\tilde\mu(\Xi)=0$ by the sub-additivity of the measure $\tilde\mu$. This means that the complement of $\Xi$, denoted by $\Xi^{{c}}$, is uncountable, otherwise the continuity and countable additivity of $\tilde{\mu}$ will yield that $\tilde{\mu}(\Xi^c)=0$, and hence  $$\tilde{\mu}(\Sigma_2^\infty)=\tilde{\mu}(\Xi)+\tilde{\mu}(\Xi^c)=0.$$ It is a contradiction.
    Consequently, for each $\bi \not \in   \Xi$, one has that, for each $p\in 2\Z+1$,  the set $ \Lambda(\pi_p(\bi))$ is a spectrum for $\mu_4$. This completes the proof of Theorem \ref{coro1}.
    \end{proof}

Combining Theorem \ref{coro1} with Theorem \cite[Theorem 4.3]{FHW2018}, we obtain the following result.
\begin{coro}
  An integer $p \in \Z$ is a spectral eigenvalue of $\mu_4$ if and only if  the set
    $$    \mathop \cap_{p\in 2\Z+1}\{\Lambda\subset \R: \text{$\Lambda$ and $p\Lambda$ are all spectra for the measure $\mu_4$}\}$$ has the cardinality of the continuum
\end{coro}
    We end this section by giving a result which is analogous to Theorem \ref{theodim}. However, they are different from each other in the case that $N(C)$ is non-empty.
    \begin{theo}\label{theodim'}
    Assume that  $C:=\{c_0,\ldots,c_{m-1}\}\equiv\{1,3\}\ppmod4$,  and we write $\cN(C)$ as in \eqref{eq.6'}. Then the following statements hold.\\
    \indent$({\rm i})$ There exists $C$ such that $\cN(C)$ is empty.\\
    \indent$({\rm ii})$ The set $\cN(C)$ is an infinite set if $\cN(C)$ is not an empty set.\\
    \indent$({\rm iii})$ There exists $C$ such that $\cN(C)$ contains  countably  many elements.\\
    \indent$({\rm iv})$ For any $\ep>0$, there exists a set $C$ such that the Hausdorff dimension of $\cN(C)$ is greater than $1-\ep$, i.e., $\dim_H\cN(C)>1-\ep.$
    \end{theo}
    \begin{proof}
    The statement  $({\rm i})$   immediately follows from Theorem \ref{theodim} $({\rm i})$ by setting $C=\{1, 7\}$. One can also refer to \cite[Example 2.6]{F2019}  for an alternative proof of $({\rm i})$.

    \smallskip
    $({\rm ii})$
     If $\cN(C)$ is not an empty set, by Proposition \ref{prop.2} and \eqref{eq.6'}, there is a label
     $$A:= a_{1 }a_{2 }a_{ 3}a_{4}  \cdots \in  \cN(C),$$
     and an integer  $\lambda\in \Lambda_I(A)\setminus \Lambda(A)$ such that $\lambda$ has the following unique infinite quasi $4-$based expansion  with respect to the labeled tree $(\mathcal{T}, A)$,
     \begin{equation}\label{eq.10'}
     \lambda\equiv \sum_{i=1}^{n}4^{i-1} \omega_{i }\ppmod{4^n},\qquad (n\ge1),
     \end{equation}
     where $\omega_{i }\in \{0, a_i\}$
     and the set  $I=\{n:\,\omega_{n }\in C\}$   has infinity many elements.       List $I$   as the form of $\{n_1,n_2,\ldots\}$ in order of size.

    Recall that the assumption $C:=\{c_0,\ldots,c_{m-1}\}\equiv\{1,3\} ~\ppmod4$ implies that for each $n  \in \N$ there exists $a'_{n }\in C$ such that $a'_{n} \not = a_{n}$ and $a'_{n} + a_{n} \equiv 0 ~~\ppmod4$. Thus, associated to $I$  and the label $A\in \cL$ above, one can construct a sequence of labels $A_\ell ~(\ell \in \N)$ as follows,
 \begin{equation*}
  A_\ell:= b_{1 }b_{2 }b_{ 3}b_{4}  \cdots \in  C^\infty,
    \end{equation*}
    where
     \begin{eqnarray*}
    b_{n }=
 \begin{cases}
 a'_{n},   &  \text{if   $n  \in I$};\\
 a_{n},  &  \text{others}.
 \end{cases}
   \end{eqnarray*}
Now we show that $L_\ell \in \cN(C)$ for all $\ell \in \N$. For this, we consider the integer
$$\lambda_\ell:=\lambda-\sum_{i=1}^{\ell}4^{i-1} \omega_{i}.$$
  One can easily check, from \eqref{eq.10'}, that
  $$ \lambda_\ell \equiv \sum_{j=1}^{n}4^{j-1}\omega'_{i}\pmod{4^n} \ \mbox{ for all } n\ge1, $$ where
  $$\omega'_{i}=\begin{cases}               0,& i  \in \{ n_1,  \ldots,  n_\ell  \};\\               \omega_{i},&\mbox{otherwise}.      \end{cases}$$
This is equivalent to that the integer $\lambda_\ell$ has an infinite quasi $4-$based expansion $\omega'_{1}\omega'_{2}\cdots     $ with respect to the labeled tree  $(\mathcal{T}, L_\ell)$, whence, $L_\ell \in \cN(C)$. From our above construction, $L_k\ne L_t$ if $k\ne t$, which concludes that $\cN(C)$ is an infinite set.

    \smallskip
    ${\rm(iii)}$
     Let $C=\{1,3\}$. We first claim that $\cN(C)$ is not an empty set. Indeed, it is easy to check that
     $$-1 \equiv \sum_{j=1}^{n}4^{j-1} 3 \quad  \pmod{4^n} \ \mbox{ for all } n\ge1.$$
     This property yields that
     $-1 \in \Lambda_I(A)\setminus \Lambda(A)$, where
     $$\Lambda(A):=\left\{\sum_{i=0}^n 4^i c_i : c_i \in \{0, 3\}\right\}.$$
    Thus, $A=333\cdots \in \cN(C)$ by Proposition \ref{prop.2}, the claim follows.

    On the other hand, since $C\subseteq \Sigma_4=\{0,1,2,3\}$, it follows from   Lemma  \ref{lembase4} that, for any label $A$,  the associated  infinite quasi $4-$based expansion $\bi=i_1i_2i_3\cdots$   of any integer must end in $3^\infty$.
    In other words,   $\cN(C)$ is a subset of the countable set
    $$\{\bi*3^\infty:\,\bi\in\Sigma_3^*\}.$$
    Thus, the desired result {\rm(iii)} follows from {\rm(ii)}.
    This completes the proof of {\rm(iii)}.

    ({\rm iv})
    For any $0<\ep<1$, we choose a positive integer $r$   such that $r^{-1}<\ep$.
    We consider the set $C=\{c_0,c_1\}$, where $c_0=1$ and $c_1=3\sum_{n=0}^r4^n$. Let
    $$\Gamma=\{a_1a_2\cdots  a_n \cdots \in C^\infty:\, a_{nr+1}=c_1, n\ge1\}.$$
    On the one hand, it follows from \cite{F1} that
    $$\dim_H\Gamma=\dfrac{r-1}r=1-\dfrac1r>1-\ep.$$
    On the other hand, for each $A \in \Gamma$, with simple calculations, one can check that
     $$-1\equiv\sum_{k=1}^n 4^{k-1}\omega_k\pmod{4^n}, \qquad\text{ for all $n \geq 1$},$$
     where
    $$\omega_k=\begin{cases}
                c_1,&k=nr +1,n\ge 0;\\
                0,&\mbox{otherwise}.
                \end{cases}.$$
    This means that $-1 \in \Lambda_I(A)\setminus \Lambda_I(A)$
    has an infinite quasi $4-$based expansion $a_1a_2\cdots$. By Proposition \ref{prop.2}, one has that $\Gamma\subset \cN(C)$, which implies the conclusion in ({\rm iv}).
    \end{proof}
     \begin{rema}
            From Theorem \ref{theodim'}, we get that the Hausdorff dimension of $N(C)$ defined as in \eqref{eq.6'} can approach 1 if we choose suitable integer subset $C$. However, we can not find an example $C$ such that $\dim_HN(C)=1$. So, we have a conjecture that $\dim_HN(C)<1$ for all finite integer subset $C$ with $C\equiv\{1,3\}\ppmod4.$
    \end{rema}

\section{Construction of eigen-spectra for middle-fourth Cantor measure}
\label{sec.eigenvalue}
     The first aim of the present section is to seek a method to construct {\it common eigen-spectrum} for $\mu_4$ corresponding to all spectral eigenvalues $p\in2\Z+1$, and the second aim is to construct an explicit   eigen-spectrum $\Lambda$ for $\mu_4$ such that its any odd scaling set is still  a spectrum for $\mu_4$.   The main result Theorem \ref{lemkey} will provide a construction principle, and Theorem \ref{thm.2} will construct such a common eigen-spectrum.

   Our key strategy is to check under what conditions the set $\Lambda(A_p)$ (see  \eqref{eq.4.1}) is a spectrum  for $\mu_4$. More exactly, by setting $p \in 2\Z+1$ and $\cL_p:=C_p^\infty,$ where $C_p=\{-p, p\},$ one has that each
    \begin{equation}\label{eq.19'}
        A_p=a_1 a_2 a_3\cdots \in \cL_p, \qquad \text{ with} \qquad a_j \in C_p,
    \end{equation}
     corresponds to a special labeled tree $(\mathcal{T}, A_p)$ described in Subsection \ref{subsec.2.2}, we will then consider
    \begin{gather}\label{eq.4.1}
    \Lambda(A_p):=\left\{\sum_{k=1}^n   4^{k-1}  \omega_k: \omega_k \in \{0, a_k\}\right\}.
    \end{gather}
       Remembering that \cite[Proposition 4.5]{DHS2009} showed that, for each label $A_1\in \cL_1$, the set $\Lambda(A_1)$ is a spectrum for $\mu_4$.
    Notice that $A_1=a_1a_2a_3 \cdots \in \cL_1$ and  $A_p=(pa_1) (p a_2)(p a_3) \cdots \in \cL_p$ imply that $\Lambda(A_p)=p\Lambda(A_1)$.
    Therefore, in order to construct a {\it common eigen-spectrum} of $\mu_4$ corresponding to all $p \in 2\Z+1$, the first and foremost issue is  to determine under what conditions the set $\Lambda(A_p)$ will be a spectrum for $\mu_4$. 

    The completeness of $\Lambda(A_p)$ defined in \eqref{eq.4.1} can be characterized as follows.

    \begin{theo}\label{lemkey}
        With the same notations above, $\Lambda(A_p)$ is not a spectrum for the measure $\mu_4$ if and only if there is an integer $\lambda$ such that the quasi $4-$based expansion of $\lambda$ with respect to the label $(\cT, A_p)$ is infinite and ultimately periodic, namely, there is a path $\omega_1\omega_2\omega_3\cdots\in\{-p,0,p\}^\infty$ without ending in $0^\infty$, such that
        $$
        \lambda\equiv\sum_{k=1}^{n} 4^{k-1}\omega_k \ppmod{4^n}
        $$
        holds for all $n \in \N$, where $\omega_k\in\{0, a_k\}, k \ge1.$
    \end{theo}
    The sufficient part of Theorem \ref{lemkey} is clear since the assumption implies that the set $\Lambda(A_p)$ is a proper subset of $\Lambda_I(A_p)$, where
    \begin{gather}
      \label{eq.4.2}
    \Lambda_I(A_p)=\Bigg\{\lambda\in \Z:\,\lambda\equiv\sum_{k=1}^{n} 4^{k-1}\omega_k \ppmod{4^n} \mbox{ for all }n\ge1,
    \omega_k\in\{0, { a_k}\}\Bigg\}.
    \end{gather}
    Consequently, $\Lambda(A_p)$ is not a spectrum for $\mu_4$, since $\Lambda_I(A_p)$ is a spectrum for $\mu_4$ by Proposition \ref{prop.2}.

     The proof of the necessary part of Theorem \ref{lemkey} needs us to
     make some preparations in the first four Subsections \ref{sec.4.1},  \ref{sec.4.2},  \ref{sec.4.3},  \ref{sec.4.4}, and then we give its detailed proof in Subsection  \ref{sec.4.5}.

    \subsection{Algebraic operations on symbolic space and their basic properties}\label{sec.4.1}

    The algebraic operations on the symbolic space $\Sigma_m^\infty$ ($2 \leq m \in \N$) will be defined in this subsection. This  is analogous to that of the ring $\Z_p$ of $p$-adic integers, where $p$ is a prime, for instance, see \cite[Chapter 1]{R2000}. We have to say that the algebraic operations on the symbolic space in this subsection maybe  well-known to many researchers, but we did not find them in the existing literature. So we provide them here for completeness.

    For any two fixed ${\bf{i}}=i_1i_2\cdots, \  {\bf{j}}=j_1j_2\cdots  \in \Sigma_{m}^{\infty}$, the {\it sum}  $\textbf{i}+\textbf{j}$  of ${\bf{i}}$ and $\bf{j}$ is defined   componentwise  using the system of carries to keep them in the range $\Sigma_m.$ More precisely, for $n\in \N$,
    the $n$th component of   $\textbf{i}+\textbf{j}$ can be computed by the following formula
    $$
    i_n+j_n+k_{n-1}-mk_n,
    $$
    where
        \begin{equation*}
    k_n=
    \begin{cases}
    0& ~{\rm if}~n=0,\\
    \left[\frac{i_{n}+j_{n}+k_{n-1}}m\right] \in\{0,1\} & ~{\rm if}~n\ge 1.
    \end{cases}
    \end{equation*}
    In particular, the first component of   $\textbf{i}+\textbf{j}$ is $i_1+j_1$ if it is less than or equal to $m-1$, or $i_1+j_1-m$ otherwise.

    For  ${\bf{i}}=i_1i_2\cdots \in \Sigma_{m}^{\infty}$ with
    $\bi \not = 0^\infty$, the {\it additive  inverse} of $\bf{i}$ is defined by
    $$-\bi=0^{n-1}(m-i_n)(m-1-i_{n+1}) (m-1-i_{n+2})\cdots \in \Sigma_{m}^{\infty},$$
    where $n=\min\{k\in\N: i_k\not=0\}$.
    Thus, it makes sense to define the  \textit{difference} of  $\bi$ and $\bj$ in $\Sigma_m^\infty$ as
    $$ \bi-\bj =\bi +(-\bj).$$
    Obviously, the symbolic space $\Sigma_m^\infty$ with addition operation is an Abelian group, in which the     {\it neutral  element} is the zero sequence $0^\infty$ since $\bi+(-\bi)=0^\infty$ holds for each $\bi \in \Sigma_{m}^{\infty}.$

    For any positive integer $a \in \N$ and $\bi=i_1i_2\cdots \in\Sigma_m^\infty$, the {\it scalar product} of $a$ and $\bi$ is defined by
    $$
    a\bi:=\underset{a \text{ times}}{\underbrace{\bi+\bi+\cdots +\bi}}.
    $$
    Similarly, for any negative integer $a \in -\N$, the {\it scalar product} of $a$ and $\bi$ is defined by
        $$
    a\bi:=\underset{-a \text{ times}}{\underbrace{-\bi-\bi-\cdots -\bi}}.
        $$
    Clearly, $-(a\bi)=(-a)\bi=a(-\bi)$ holds for all $a \in \N$.
    Moreover,  $\bj \in\Sigma_m^\infty$ is called the \textit{quotient} of $\bi\in \Sigma_m^\infty $ and a non-zero integer $a \in \Z$  if $a\bj=\bi,$   denote it by  $\bj:=\frac1a\bi$ or $\bj:=a^{-1}\bi$.

    We provide an example to illustrate the algebraic operations in terms of $\Sigma_4^\infty$.
    \begin{exam}
    The sum of $2130\cdots$ and $3211\cdots$ in $\Sigma_4^\infty$ is displayed as follows:
    \begin{equation}\label{1}
        \begin{matrix}
            &2&  &1&  &3 &  & 0&&\cdots \\
           +&3&_1&2&_1&1 &_1& 1&&\cdots\\ \hline
          = &1 &  &0 &  & 1&& 2&&\cdots
        \end{matrix}
    \end{equation}
    The difference of $2130\cdots$ and $3211\cdots$ is displayed as follows:
    \begin{equation}\label{1}
        \begin{matrix}
          & 2 & &1 & & 3 &&0&&\cdots \\
        - & 3 & &2 & & 1 &&1&&\cdots\\  \hline
        + &1&_0&1&_0&2&_1&2&&\cdots\\ \hline
         =& 3 &  &2 &  & 1 &   &3&&\cdots
        \end{matrix}
    \end{equation}
    The product of $3$ and $32^\infty$ is displayed as follows:
    $$
    \begin{matrix}
    &3& &2&&2&&2&&\cdots\\
    \times&3&\quad_2&&\quad_2&&\quad_2&\\ \hline
    =&1&&0&&0&&0&&\cdots
    \end{matrix}
    $$
    So, $\frac13(10^\infty)$ coincides with $3 2^\infty$.
    \end{exam}
    The basic properties of algebraic operations of $\Sigma_m^\infty$ are stated as follows. These will be used to prove Lemma \ref{lemisum} below and Theorem \ref{lemkey}.

        \begin{prop}\label{prop1}
    With the algebraic operations on $\Sigma_m^\infty$, the following statements hold. \\
     \indent  $({\rm i})$
    Let $a\in \Z$ be a non-zero integer and
    ${\bf{i}}=i_1i_2\cdots \in\Sigma_m^\infty$. Then, for each $n \in \N$, the $n$th component of the scalar product $a\bi$ is the unique number $x_n\in\Sigma_m$ such that
    $$x_n=ai_n+k_{n-1}-mk_n,$$
    where $k_0=0$ and $k_n=[m^{-1}(ai_n+k_{n-1})]$. \\
     \indent  $({\rm ii})$
    The distributive law of scalar product on $\Sigma_m^\infty$ holds, i.e., for $k,l\in\Z$ and for $\bi,\bj\in\Sigma_m^\infty$,
    $$k(\bi+\bj)=k\bi+k\bj,\quad (k+l)\bi=k\bi+l\bi.$$
    \indent  $({\rm iii})$  For any positive integer $k\in\N$ and $\bi\in \Sigma_m^\infty,$ one has $m^k\bi=0^k*\bi$. \\
     \indent $({\rm iv})$ Let $\rho$ be the metric on $\Sigma_m^\infty$ defined by
            \begin{equation}\label{eq.12}
                \rho(\bi,\bj)=m^{-\inf\{n:\,i_n\neq j_n\}},
            \end{equation}
     where   $ \bi=i_1i_2\cdots, \bj=j_1j_2\cdots\in\Sigma_m^\infty$. Then for any $\bi, \ \bj,  \ \widetilde{\bi}, \ \widetilde{\bj} \in\Sigma_m^\infty$ with $\rho(\widetilde{\bi}, \ \widetilde{\bj})<\rho(\bi, \ \bj)$, one has that
            $$\rho(\bi+\widetilde{\bi}, ~ \bj+\widetilde{\bj})=\rho(\bi,\bj).$$
     \indent  $({\rm v})$  Let $a\in 2\Z+1$.
        Then for each ${\bf{i}}=i_1i_2\cdots \in\Sigma_m^\infty$, the quotient of $\bi$ and $a$ exists and is unique. Consequently, the following map
            \begin{equation}\label{eq.fa}
                f_a(\bi)=a\bi
            \end{equation}
            is a bijection from $\Sigma_m^\infty$ onto $\Sigma_m^\infty$.
        \end{prop}
    \begin{proof}
        (i), (ii) and (iii) are immediately obtained by the definitions of algebraic operations on $\Sigma_m^\infty$.

        Let $i_n, j_n, \widetilde{i_n}, \widetilde{j_n}$ be the $n$th components of $\bi, \ \bj,  \ \widetilde{\bi}, \ \widetilde{\bj}$, respectively. Without loss of generality, we assume that $\rho(\bi, \ \bj)=m^{-k}$ and $\rho(\widetilde{\bi}, \ \widetilde{\bj})=m^{-\ell}$, where $k<\ell$ and
            $$k=\min\{n:\, i_n\ne j_n\},\quad \ell=\min\{n:\, \widetilde{i_n}\ne \widetilde{j_n}\}.$$
        Since $i_n+ \widetilde{i_n}=j_n+ \widetilde{j_n} $ holds for all $n<k$ and $i_k+ \widetilde{i_k}\ne j_k+ \widetilde{j_k},$ it follows from the  definition of   $\rho$  in \eqref{eq.12} that $\rho(\bi+\widetilde{\bi}, ~ \bj+\widetilde{\bj})=m^{-k},$  $({\rm iv})$ follows.

        For $({\rm v})$, it suffices to prove that there is a unique $\bj\in \Sigma_m^\infty$ such that $a\bj=\bi$ for the positive integer $a\in\N$ and $\bi\in\Sigma_m^\infty$. This is because $-(a\bj)=(-a)\bj=a(-\bj)$ holds for all $a \in \N$ and all $\bj\in \Sigma_m^\infty$.

          We first consider the case $\bi=0^\infty$. Assume $\bj=j_1j_2\cdots\in\Sigma_m^\infty$ makes $a\bj=0^\infty$. By the definition of scalar product or addition on $\Sigma_m^\infty$, we observe
            $$aj_1\equiv0\pmod m,$$
        which implies $j_1\equiv0\ppmod m$ since $a,m$ are coprime. The fact $j_1\in\{0,1,\ldots,m-1\}$ says that $j_1=0$. So, by induction, we get $\bj=0^\infty$.

        Next, we consider other $\bi=i_1i_2\cdots \in \Sigma_m^\infty$.  We will define each component $j_n$ of $\bj=j_1j_2\cdots$ by the recursion such that $a\bj=\bi$. More precisely, since $\gcd(a,m)=1$, by the division operations, there exists a   pair $(j_1,k_1)$ with $j_1\in\Sigma_m$ and $k_1\in\{0,1,\ldots,a-1\}$ such that
            $$aj_1=mk_1+i_1.$$
        Assuming that $j_n$ and $k_n$ have been defined. Then there exists a   pair $(j_{n+1},k_{n+1})$ with $j_{n+1}\in\{0,1,2,3\}$ and $ k_{n+1} \in\{0,1,\ldots, a-1\}$ such that
        $$aj_{n+1}+k_n=mk_{n+1}+i_{n+1},$$
    holds by using division operations again. Consequently, we get a desired  infinite path $\bj=j_1j_2\cdots\in\Sigma_m^\infty$ such that $a\bj=\bi.$  {If there exists another $\bj'=j_1'j_2'\cdots\in\Sigma_m^\infty$ such that $a\bj'=\bi$. Then $a(\bj-\bj')=a\bj-a\bj'=\bi-\bi=0^\infty$. So, $\bj-\bj'=0^\infty$, as we showed before. The uniqueness of $\bj$ is obtained.}
  \end{proof}

\subsection{Algebraic operations on the set of ultimately periodic sequence}\label{sec.4.2}

    We call an infinite word $\bi\in\Sigma_m^\infty$ is \textit{ultimately periodic} if there exist $\bi_1,\bi_2\in\Sigma_m^*$ such that $\bi=\bi_1 \bi_2^\infty$. Also, $\bi_1$ is called the aperiodic part and $\bi_2$ is the periodic part of $\bi$. We use the symbol $\Q_m$ to denote the collection of all sequence in $\Sigma_m^\infty$ that are ultimately periodic. In this section, we are especially concerned about the algebraic operations on the set of ultimately periodic sequence  in $\Sigma_4^\infty$.

    A simple observation is stated as follows.
    \begin{prop}\label{lemtwoperiodic}
    Let $\bi$ and $\bj$ be two ultimately periodic sequence in  $\Sigma_4^\infty$. Then their aperiodic parts and periodic parts can be chosen such that they are of the same length.
    \end{prop}
    \begin{proof}
       Write $\bi=\bi_1 \bi_2^\infty$ and $\bj=\bj_1 \bj_2^\infty$, where $\bi_k, \bj_k \in \Sigma_4^*, k=1,2.$  We choose a positive integer $\ell$ large enough such that $s=\ell \cdot |\bi_2|\cdot |\bj_2|>\max\{|\bi_1|,|\bj_1|\}$. By writing $\bi|_s=\bi_1'$, $\bj|_{s}=\bj_1'$ and  $\bi|_{2s}=\bi_1'\bi_2',  \bj|_{2s}=\bj_1'\bj_2',$ it follows that $|\bi_1'|=|\bi_2'|=|\bj_1'|=|\bj_2'|=s$  and hence $\bi=\bi_1' \bi_2'^\infty$, $\bj=\bj_1 \bj_2'^\infty$. We get  the desired conclusion.
    \end{proof}

    Let $s$ be a positive integer with $s \ge1$ and let $\bi=i_1i_2\cdots\in\Sigma_4^\infty$.  We define the map $\pi_s: \Sigma_4^\infty \rightarrow \Sigma_{4^s}^\infty$ as $\pi_s(\bi)=\sigma:=\sigma_1\sigma_2\cdots \in \Sigma_{4^s}^\infty$, where
        $$\sigma_k:=\sum_{i=1}^s 4^{i-1}i_{(k-1)s+i}\in\Sigma_{4^s}.$$
    It is clear that the map $\pi_s: \Sigma_4^\infty \rightarrow \Sigma_{4^s}^\infty$ is a bijection, and the restriction of $\pi_s$ on $\Q_4$ is also a bijection from $\Q_4$ to $\Q_{4^s}$.

    The next lemma  shows that $\Q_4$ is closed under the algebraic operations on $\Sigma_4^\infty$.
        \begin{lem}\label{lem.1}
            Let $a$ be a positive odd integer, and let $\bi$ and $\bj$ be two ultimately periodic sequence. Then $({\rm i})$ $\bi + \bj \in \Q_4$, and $({\rm ii})$ $\frac{\bi}{a} \in \Q_4.$
        \end{lem}
        \begin{proof}
        $({\rm i})$ If $\bi,\bj $ are  ultimately periodic sequence in $\Sigma_4^\infty$, it  follows from Proposition \ref{lemtwoperiodic} that there is a positive integer $s$ such that $\bi_1, \bi_2, \bj_1, \bj_2 \in \Sigma_4^s$ and $\bi =\bi_1 \bi_2^\infty$ and $\bj=\bj_1  \bj_2^\infty$.  Write
            $$ \pi_s(\bi)=\sigma_1\sigma_2^\infty \in \Sigma_{4^s}^\infty  \quad \text{and} \quad \pi_s(\bj)=\tau_1\tau_2^\infty\in \Sigma_{4^s}^\infty,$$
        where $\sigma_i, \tau_i\in \Sigma_{4^s}$ for $i=1,2$. By the definition of additional operations in $\Sigma_{4^s}^\infty$,  the $n$th ($n\ge2$) component $\gamma_n$ of $\gamma:=\pi_s(\bi + \bj)$ can be obtained as follows:
        \begin{equation}\label{eqgamma}
            \gamma_n=\sigma_2+\tau_2  +k_{n-1} -4^s\left[\frac{\sigma_2+\tau_2+k_{n-1}}{4^s}\right],
        \end{equation}
        where $k_1=\big[\frac{\sigma_1+\tau_1}{4^s}\big]$ and     $k_n=\big[\frac{\sigma_2+\tau_2+k_{n-1}}{4^s}\big]\in\{0,1\}$ for all $n\ge2$. Thus, the following statements hold.
        \begin{itemize}
            \item If $\sigma_2+\tau_2<4^s-1$, then $\gamma_n=\sigma_2+\tau_2$ holds for all $n\geq 3$;
            \item If $\sigma_2+\tau_2=4^s-1$, then $\gamma_n=0$ (resp. $4^s-1$) holds for all $n\ge2$ if $k_1=1$ (resp. $k_1=0$);
            \item If $\sigma_2+\tau_2\ge 4^s$, then $\gamma_n=\sigma_2+\tau_2+1-4^s$ holds for all $n\geq 3$.
        \end{itemize}
    In a word, we have proved that $\gamma_n=\gamma_{n+1}=\cdots$ for all $n\ge3$, which means that $\bi+\bj=\pi_s^{-1}(\gamma)$ is an ultimately periodic sequence  in $\Sigma_4^\infty$.

        $({\rm ii})$ Since $\bi$ is an ultimately periodic sequence in $\Sigma_4^\infty$, there is a positive integer $s$ such that $\bi =\bi_1 \bi_2^\infty$ by Proposition \ref{lemtwoperiodic}, where $\bi_1, \bi_2  \in \Sigma_4^s$. Also, for the positive odd integer $a$, we choose a positive integer $s$   large enough such that  $a<4^s$. Put $\sigma=\pi_s(\bi)=\sigma_1 \sigma_2^\infty$.
        Notice that $\pi_s(\frac1a\bi)=  \frac1a\sigma  \in\Sigma_{4^s}^\infty$. Thus, it suffices to prove that   $\frac1a\sigma$ is ultimately periodic in $\Sigma_{4^s}^\infty.$

        Indeed, by setting $\theta=\frac1a\sigma= \theta_1\theta_2\cdots \in \Sigma_{4^s}^\infty,$  one gets that $a\theta=\sigma$, where the $n$th component $\theta_n$ of $\theta$ satisfies that
        \begin{equation}\label{eqathetan}
            a\theta_n+k_{n-1}-k_n4^s=\sigma_2,
        \end{equation}
    where $k_1=[4^{-s}a\theta_1]$ and $k_n=[4^{-s}(a\theta_n+k_{n-1})]\in  \{0,1,\ldots, a -1\}$ for $n \geq 2$.

    Since  $\{k_n\}_{n=2}^\infty$ is bounded, there exist positive integers $n_0$ and $N\ge2$ such that $k_{n_0+N}=k_{n_0}$. This, combined with \eqref{eqathetan}, yields that
            $$a\theta_{n_0+N+1}-\sigma_2\equiv  a\theta_{n_0+1}-\sigma_2\pmod{4^s}.$$
    It follows from $\gcd(a,4)=1$ that    $$\theta_{n_0+1}\equiv\theta_{n_0+N+1}\pmod{4^s},$$
    which yields that $\theta_{n_0+1}=\theta_{n_0+N+1}$ since all $\theta_n\in\{0,\ldots,4^s-1\}$.
    By \eqref{eqathetan} again, one obtains that $k_{n_0+N+1}=k_{n_0+1}$. Continuing the above procedure, we will get $\theta_{n_0+N+k}=\theta_{n_0+k}$ holds for all $k\ge1$. Therefore $\theta$ is ultimately periodic.  The proof of Lemma \ref{lem.1} is completed.
    \end{proof}

    \subsection{A formal infinite series in symbolic space}\label{sec.4.3}

        The following  lemma   generalizes the representation in Proposition \ref{prop.1} to that of  a class of formal infinite series on $\Sigma_4^\infty$, which is crucial to  Lemma  \ref{lempi2} below.
        \begin{lem}\label{lemisum}
    Given a sequence of infinite sequence $\bi_k=i_1^{(k)} i_2^{(k)}i_3^{(k)} \cdots \in\Sigma_4^\infty$ which ends in $0^\infty$ or $3^\infty$, then the formal infinite series $\sum_{k=1}^\infty4^{k-1}\bi_k$ is well-defined in $\Sigma_4^\infty$ and its $n$th  component $i_n$ is represented as
        \begin{equation}\label{eqtn}
            i_n=\left[\dfrac{d_n}{4^{n-1}}\right]-4\left[\dfrac{d_n}{4^{n}}\right],
        \end{equation}
    where each $d_n$ is an integer defined by $$d_n=\sum_{k=1}^n\sum_{j=1}^{n+1-k}4^{j+k-2}i^{(k)}_j.$$
    \end{lem}
        \begin{proof}
    For each $n \in \N$, we define the $n$th partial sum of $\sum_{k\ge1}4^{k-1}\bi_k$ by $$\theta_n:=\sum_{k=1}^n 4^{k-1}\bi_k \in \Sigma_4^\infty.$$
    By {\rm(iii)} in Proposition \ref{prop1}, one has  that  $\theta_n=\sum_{k=1}^n0^{k-1}*\bi_k$, which yields that the $n$th component of  $\theta_k, k \geq n,$ are the same by the definition of the {\it sum} operation on $\Sigma_4^\infty$. Whence, it follows from \eqref{eq.12} that
        $$\rho(\theta_k,\theta_{l})\le 4^{-n} \qquad (\forall k,l \geq n).$$
    In other words,  $\{\theta_n\}_{n\ge1}$ is a Cauchy sequence in $\Sigma_4^\infty$ in the metric $\rho$ as in \eqref{eq.12}.
    Thus, the formal infinite series $ \sum_{k=1}^\infty4^{k-1}\bi_k$ is well-defined since $\Sigma_4^\infty$ is a compact metric space. Define
        $$\bi:= \sum_{k=1}^\infty4^{k-1}\bi_k=\underset{n\to\infty}{\lim}\theta_n\in\Sigma_4^\infty.$$
    Write  $\bi=i_1i_2\cdots \in \Sigma_4^\infty$.
    Note that for each $n \in \N$ the $n$th component $i_n$ of $\bi$ is the same as that of $\theta_n$.
    However, the infinite sequence $\theta_n$ is obvious the $4-$based  expansion of the integer $\lambda_n$, where
    $$
    \lambda_n=\sum_{j=1}^\infty 4^{j-1}i^{(1)}_j+ 4\sum_{j=1}^\infty 4^{j-1}i^{(2)}_j+\cdots+ 4^{n-2}\sum_{j=1}^\infty 4^{j-1} i^{(n-1)}_j +4^{n-1}\sum_{j=1}^\infty i^{(n)}_j=:\sum_{k=1}^n\sum_{j=1}^\infty4^{j+k-2}i^{(k)}_j,
    $$
    from which one gets, by the first part of Proposition \ref{prop.1}, that
    $$i_n=\left[\dfrac{\lambda_n}{4^{n-1}}\right] -4\times\left[\dfrac{\lambda_n}{4^n}\right].$$
    Therefore, we get the desired result \eqref{eqtn} by the second part of Proposition \ref{prop.1}, since $\lambda_n\equiv d_n\ppmod{4^n}$, where $d_n$ is defined by
    $$
    d_n=\sum_{j=1}^n 4^{j-1}i^{(1)}_j+ 4\sum_{j=1}^{n-1} 4^{j-1}i^{(2)}_j+\cdots+ 4^{n-2}\sum_{j=1}^2 4^{j-1} i^{(n-1)}_j +4^{n-1}i^{(n)}_1=:\sum_{k=1}^n\sum_{j=1}^{n+1-k}4^{j+k-2}i^{(k)}_j.
    $$
    This finishes the proof of Lemma \ref{lemisum}.
    \end{proof}

\subsection{An important  Transformation}\label{sec.4.4}

  By Proposition \ref{lembase4},   there is a canonical mapping $\Pi:\,\Z\mapsto\Sigma_4^\infty$ defined by
    \begin{equation}\label{eq.13}
     \Pi(\lambda)=i_1i_2\cdots,
    \end{equation}
    where $\lambda=\sum_{n=0}^\infty4^ni_n\in\Z$ and $i_1i_2\cdots$ is the unique $4-$based expansion of $\lambda$. Proposition \ref{prop1} (ii) implies that the canonical mapping $\Pi$ in \eqref{eq.13} is linear, that is, for $k_1,k_2,\lambda_1,\lambda_2\in\Z$,
      \begin{equation}\label{eq.20}
         \Pi(k_1\lambda_1+k_2\lambda_2)=k_1\Pi(\lambda_1)+k_2\Pi(\lambda_2).
     \end{equation}
    For each $p \in 2\Z+1$, we set $\Omega_p=\{-p,0,p\}$ and consider  the transformation $h_p: \Omega_p^\infty \rightarrow \Sigma_4^\infty$ as follows,
    \begin{equation}\label{eqha}
    h_p(\omega)=\sum_{k=1}^\infty \Pi(4^{k-1}\omega_k), \qquad
    \omega=\omega_1\omega_2\cdots \in\Omega_p^\infty.
        \end{equation}
    It is well-defined since \eqref{eq.20} and Proposition   \ref{prop1} (iii) imply that
    $$\Pi(4^{k-1}\omega_k)=4^{k-1}\Pi(\omega_k)=0^{k-1}\ast \Pi(\omega_k)\in \Sigma_4^\infty, \qquad (\forall k \in \N),$$
    which yields that the series on the right hand side of \eqref{eqha}
    converges in $\Sigma_4^\infty$ by Lemma \ref{lemisum}.

    \begin{lem}\label{lempi2}
    If the mapping $h_p: \Omega_p^\infty \rightarrow \Sigma_4^\infty$ is defined as in \eqref{eqha}, where $p \in 2\Z+1$, then $h_p$ is injective and
        \begin{equation}\label{eqah1}
            h_p(\omega)=ph_1\Big(p^{-1}\omega\Big), 
        \end{equation}
    where $\omega=\omega_1\omega_2\cdots \in\Omega_p^\infty$ and  $p^{-1}\omega:=(p^{-1}\omega_1)*(p^{-1}\omega_2)*\cdots\in\Omega_1^\infty$.
        \end{lem}
        \begin{proof}
        We firstly claim that
        \begin{equation}\label{eq.14}
            \rho\big(\Pi(p),\pi(0)\big)=\rho\big(\Pi(-p),\Pi(0)\big) =\rho\big(\Pi(-p),\Pi(p)\big)=4^{-1},
        \end{equation}
    where $\rho$ is the metric on $\Sigma_4^\infty$ defined as in \eqref{eq.12}.
    In fact, since $p \in 2\Z+1$, it follows that the $4-$based expansion of $p$
        $$\Pi(p)=i_1i_2i_3\cdots \in \Sigma_4^\infty$$
    ends in $0^\infty$ or $3^\infty$, where $i_1\in \{1,3\}$. Clearly, $\Pi(-p)=(4-i_1)(3-i_2)(3-i_3)\cdots \in\Sigma_4^\infty$ and $\pi(0)= 0^\infty$. Thus, the claim follows from the definition of $\rho$.

    We next show the map $h_p: \Omega_p^\infty \rightarrow \Sigma_4^\infty$ is injective by showing that $\rho\big(h_p(\omega^{(1)}),h_p(\omega^{(2)})\big)>0$ for any two  distinct sequence
        $$\omega^{(i)}= \omega_1^{(i)}\omega_2^{(i)}\omega_3^{(i)}\cdots \in \Omega_p^\infty, \qquad \text{where} \quad i=1,2.$$

        In fact,  by setting  $k$ be the first integer such that $\omega_k^{(1)}\not =\omega_k^{(2)}$, one gets that
        \begin{eqnarray*}
            \rho\big(h_p(\omega^{(1)}),h_p(\omega^{(2)})\big)&=&\rho\left(\sum_{j=k}^\infty \Pi(4^{j-1}\omega^{(1)}_j), \sum_{j=k}^\infty \Pi(4^{j-1}\omega^{(2)}_j)\right)\\
            &=&\rho\Big(\Pi(4^{k-1}\omega^{(1)}_k), \Pi(4^{k-1}\omega^{(2)}_k) \Big)\\
            &=&4^{-k+1}\rho\Big( \Pi(\omega^{(1)}_k), \Pi(\omega^{(2)}_k) \Big)=4^{-k}>0.
        \end{eqnarray*}
    Here,  the first two equalities follow from Lemma \ref{lemisum}, and the third equality follows from the fact that
        $$\rho\big(\Pi(4^{n-1}\omega^{(1)}_n), \Pi(4^{n-1}\omega^{(2)}_n)\big) =4^{-n+1}\rho\big(\Pi(\omega^{(1)}_n), \Pi(\omega^{(2)}_n)\big), \qquad (\forall n \in \N),$$
        and the last equality follows from \eqref{eq.14}. Therefore, $h_p: \Omega_p^\infty \rightarrow \Sigma_4^\infty$  is injective.

    At last we show that \eqref{eqah1} holds. Indeed, since
        $$p \Pi  (1)=\Pi(p),  \quad -p\Pi(1)=\Pi(-p)  \quad \hbox{and} \quad p\Pi( 4^{k-1})=\Pi(4^{k-1}p),$$
    always hold for all $p \in 2\Z+1$ and $p=0$, it follows from Proposition \ref{prop1} {\rm(ii)} and \eqref{eq.20} that, for any infinite sequence $\omega=\omega_1\omega_2\cdots \in\Omega_p^\infty$,
        \begin{equation}\label{eqdistr}
            \sum_{k=1}^n  \Pi\Big(  4^{k-1}\omega_k\Big)=\sum_{k=1}^n p\Pi\Big( 4^{k-1}\frac{\omega_k}{p}\Big)= p\sum_{k=1}^n \Pi\Big( 4^{k-1}\frac{\omega_k}{p}\Big) \qquad (\forall n\ge1).
        \end{equation}
    Letting $n \rightarrow \infty$, we get from \eqref{eqha} that
        $$h_p(\omega) =\sum_{k=1}^\infty \Pi(4^{k-1}\omega_k)
            =p\sum_{k=1}^\infty  \Pi \Big(\frac1p 4^{k-1}\omega_k\Big)=ph_1\Big(\frac1p\omega\Big).$$
    This completes the proof of Lemma \ref{lempi2}.
    \end{proof}
    \subsection{Proof of Theorem  \ref{lemkey}}\label{sec.4.5}
    Now, we have all ingredients for the proof of the necessary part of Theorem \ref{lemkey}.

    \begin{proof}[The proof of the necessary part of Theorem  \ref{lemkey}]
    By Proposition \ref{prop.2}, we can suppose that there is a  non-zero integer $\lambda\in\Lambda_I(A_p)\setminus\Lambda(A_p)$.
    Assume that the infinite quasi $4-$based expansion of $\lambda$ is $\omega=\omega_1\omega_2\cdots \in\Omega_p^\infty$ without ending in $0^\infty$, namely,
    \begin{equation}\label{eq.thm.4.1.1}
    \lambda\equiv\sum_{k=1}^n4^{k-1}\omega_k\pmod{4^n},  \qquad \text{ for all $n \ge 1$}.
     \end{equation}
     On the other hand, by Proposition \ref{lembase4}, the integer $\lambda$ has the $4-$based  expansion $\bi =i_1i_2i_3\cdots\in\Sigma_4^\infty$ which ends in $0^\infty$ or $3^\infty$, namely,
     \begin{equation}\label{eq.thm.4.1.2}
          \lambda=\sum_{k=1}^\infty 4^{k-1}i_k.
     \end{equation}

    \smallskip
 \noindent \textbf{Claim 1.}  With the same notations above, one has   $\bi=h_p(\omega)$, where $h_p:\Omega_p^\infty \rightarrow \Sigma_4^\infty$ is defined by \eqref{eqha}.

    \smallskip
    \noindent\textit{Proof of Claim 1.}  First, we assume that the $4-$based expansion of $\omega_k,~ k \geq 1,$ is the infinite sequence   $\bi_k=i_1^{(k)} i_2^{(k)}i_3^{(k)} \cdots \in\Sigma_4^\infty$ which ends in $0^\infty$ or $3^\infty$, that is $\Pi(\omega_k)=\bi_k$. This yields that $\Pi(4^{k-1}\omega_k)=4^{k-1}\bi_k$ by Proposition \ref{prop1} (iii), and hence  Lemma \ref{lemisum} implies that
    $$h_p(\omega)=\sum_{k=1}^\infty \Pi(4^{k-1}\omega_n)
    =\sum_{k=1}^\infty  4^{k-1}\bi_k \in \Sigma_4^\infty.$$
    Write
    $h_p(\omega)=j_1j_2j_3\cdots\in \Sigma_4^\infty.$ It follows from Lemma \ref{lemisum} that for each $n \in \N$ the $n$th component $j_n$ of $h_p(\omega)$ is
     \begin{equation*}
    j_n=\left[\dfrac{d_n}{4^{n-1}}\right]-4\left[\dfrac{d_n}{4^{n}}\right],
    \end{equation*}
    where
    $$d_n=\sum_{k=1}^n\sum_{j=1}^{n+1-k}4^{j+k-2}i^{(k)}_j.$$
     Because $\omega_k=\sum_{j=1}^\infty 4^{j-1}i_j^{(k)}$ for all $k \in \N$, it is clear that
    $$d_n \equiv \sum_{k=1}^n4^{k-1}\omega_k
     \pmod{4^n}, \qquad (n \ge 1).$$
    Applying Proposition \ref{prop.1} to \eqref{eq.thm.4.1.1} and  \eqref{eq.thm.4.1.2}, we get that $i_n=j_n$ for all $n \in \N$.
    Thus, $\bi=h_p(\omega)$, the claim is true. This completes the proof of Claim 1.

\smallskip

    Note that the $4-$based expansion $\bi=i_1i_2i_3\cdots\in \Sigma_4^\infty$ of the integer $\lambda$ is ultimately periodic.
    It follows from Lemma \ref{lem.1}, Claim 1 and  \eqref{eqah1} that $h_1\big(p^{-1}\omega\big)$ is ultimately periodic in $\Sigma_4^\infty$, where
    $$p^{-1}\omega :=(p^{-1}w_1)*(p^{-1}w_2)*\cdots\in\Omega_1^\infty.$$
    Consequently, $h_1\big(p^{-1}\omega\big)+1^\infty$ is also ultimately periodic in $\Sigma_4^\infty$ by Lemma \ref{lem.1} again.

    Define $b_k=p^{-1}{\omega_k} \in \{-1, 0, 1\}$. It follows from \eqref{eqha}, Lemma \ref{lemisum} and the fact $ 1^\infty=\sum_{k\ge1}\Pi(4^{k-1})$ that
    $$h_1\big(p^{-1}\omega\big)+1^\infty
    =\sum_{k=1}^\infty \Pi\big(4^{k-1} b_k \big)+\sum_{k\ge1}\Pi(4^{k-1})
    =\sum_{k=1}^\infty \Pi\big(4^{k-1}(b_k+1)\big),$$
    which is ultimately periodic in $\Sigma_4^\infty$ by Lemma \ref{lem.1}.
    Observing that $b_k+1\in\{0,1,2\}$,  from which one has
    $$\Pi\big(4^{k-1}(b_k+1)\big)=0^{k-1}*(b_k+1)*0^\infty \in \Sigma_4^\infty,$$
    and hence
    $$\sum_{k=1}^\infty \Pi\big(4^{k-1}(b_k+1)\big)=(b_1+1) (b_2+1)\cdots$$
    Consequently, both the infinite sequence  $b_1b_2\cdots\in\Omega_1^\infty$ and  $\omega=\omega_1\omega_2\cdots \in\Omega_p^\infty$ are ultimately periodic.
    The proof of Theorem \ref{lemkey} is finished.
 \end{proof}
    \begin{rema}
      Theorem \ref{lemkey}  provides the most useful way to construct common eigen-spectra for all odd spectral eigenvalues of $\mu_4$, which greatly improves the method of constructing common eigen-spectra for $\mu_4$ in    \cite[Section 4]{FHW2018}  and \cite[Section 3]{F2019}.
      In fact, the methods in \cite{FHW2018,F2019} can only construct common eigen-spectra corresponding to finitely many spectral eigenvalues of $\mu_4$, for instance, see Theorem 4.13 and Corollary 4.14 in  \cite{FHW2018}, while Theorem \ref{lemkey}  can deal with those  corresponding to all spectral eigenvalues $p \in 2\Z+1$.
    \end{rema}

    We end this paper by constructing an  explicit  eigen-spectrum of $\mu_4$ such that Theorem \ref{coro1} holds, we state it as Theorem \ref{thm.2} below. 
    \begin{theo}\label{thm.2}
        Let $A_1=a_1a_2a_3\cdots$ in \eqref{eq.19'} be as $A_1=\tau_1\tau_2\cdots$, where $\tau_n=1^n*(-1)^{n+1}$ for all $n\ge1$. Then for all odd integers $p\in2\Z+1$, all of the following sets
        \begin{equation*}
            p\Lambda(A_1)=p\left\{\sum_{k=1}^n   4^{k-1}  \omega_k: \omega_k \in \{0, a _k\}\right\},
        \end{equation*}
        are spectra for $\mu_4$.
        \end{theo}
        \begin{proof}
    Under the assumption on $A_1$ in Theorem \ref{thm.2}, for each odd integer $p \in 2\Z+1$, we set
        $$A_p:=(p a_1)(p a_2)(p a_3)\cdots\in\{-p, p\}^\infty.$$
    Clearly, $\Lambda(A_p)=p\Lambda(A_1)$.
    Suppose on the contrary that there is an odd integer $p\in 2 \Z+1$ such that the set $\Lambda( A_p)$
    is not a spectrum for $\mu_4$. By Theorem \ref{lemkey}, there exists an integer $\lambda$ satisfies that        $$\lambda\equiv\sum_{k=1}^n4^{k-1}\omega_k\ppmod{4^n}\quad\text{ for all } n\in\N,\,\omega_k\in\{0, p a_k\},$$
    where the infinite quasi $4-$based expansion $\omega_1\omega_2\cdots\in \{-p,0,p\}^\infty$ of $\lambda$ is ultimately periodic and $\#\{n\in\N:\,\omega_n\ne0\}=\infty$. Therefore,
    there are   $s, t \in \N$ such that $\omega_{s+1}\in \{-p,p\}$ and
    $$\omega_1\omega_2\omega_3\cdots=\omega_1\cdots \omega_s (\omega_{s+1}\cdots \omega_{s+t})^\infty,$$
    which yields that either $\omega_{mt+s+1}=\omega_{s+1}=p$ or $\omega_{mt+s+1}=\omega_{s+1}=-p$ holds  for all $m\ge0$. It follows that either $a_{mt+s+1}=a_{s+1}=1$ or $a_{mt+s+1}=a_{s+1}=-1$ for all $m\ge1$. However, it is impossible by the structure of the label $A_1$.  The proof of Theorem \ref{thm.2} is completed.
    \end{proof}

\noindent\textbf{Acknowledgements.} Guotai Deng is supported by the National Natural Science Foundation of China grant (Nos. 12071171), Hubei Provincial Natural Science Foundation of China (2020CFB833), and Science Foundation of Jiangxi Education Department (GJJ202302). Yan-Song Fu is supported by the National Natural Science Foundation of China (Nos.12371090) and the Fundamental Research Funds for the Central Universities (2023ZKPYLX01).


\end{document}